\documentclass{article}[12pt]

\usepackage[a4paper,top=4.0cm,bottom=2.5cm,left=2.5cm,right=2.5cm]{geometry}
\usepackage{amsthm,amsmath,amsfonts,mathrsfs}
\usepackage{epsfig,graphicx}

\theoremstyle{plain}
	
	\newtheorem*{Theorem*}{Theorem}
	
	\newtheorem*{Lemma*}{Lemma}
	
	\newtheorem*{Proposition*}{Proposition}
	
	\newtheorem*{Satz*}{Satz}

	\newtheorem*{Corollary*}{\corollaryname}
	
	\newtheorem*{Korollar*}{\corollaryname}

\theoremstyle{definition}
	
	\newtheorem*{Algorithm*}{\algorithmname}
	
	\newtheorem*{Definition*}{Definition}

	\newtheorem*{Conjecture*}{\conjecturename}
	
	\newtheorem*{Vermutung*}{\conjecturename}

	\newtheorem*{Assumption*}{\assumptionname}
	
	\newtheorem*{Annahme*}{\assumptionname}

	\newtheorem*{Conclusion*}{\conclusionname}
	
	\newtheorem*{Folgerung*}{\conclusionname}

	\newtheorem*{Hilfssatz*}{Hilfssatz}

	\newtheorem*{Feststellung*}{Feststellung}

	\newtheorem*{Example*}{Example}
	
	\newtheorem*{Beispiel*}{\examplename}

	\newtheorem*{Hypothesis*}{\hypothesisname}
	
	\newtheorem*{Hypothese*}{\hypothesisname}

	\newtheorem*{Property*}{\propertyname}
	
	\newtheorem*{Eigenschaft*}{\propertyname}

	\newtheorem*{Remark*}{\remarkname}
	
	\newtheorem*{Aussage*}{\remarkname}

	\newtheorem*{Aufgabe*}{\exercisename}
	
	\newtheorem*{Uebung*}{\exercisename}
	
	\newtheorem*{Exercise*}{\exercisename}

\makeindex

\begin{document}

% The file intro.tex should consist of:
%
%   \begin{Preface}
%     ...the preface...
%   \end{Preface}
%   \tableofcontents

%\include{intro}

\title{Ensemble filter techniques for intermittent 
  data assimilation - a survey}

\author{Sebastian Reich \\Department of Mathematics, University
    of Potsdam \\ \\ and \\ \\ Colin Cotter\\
Department of Aeronautics, Imperial College London}

\maketitle

\begin{abstract}
This survey paper is written with the intention of giving a
mathematical introduction to filtering techniques for intermittent
data assimilation, and to survey some recent advances in the
field. The paper is divided into three parts.  The first part
introduces Bayesian statistics and its application to statistical
inference and estimation. Basic aspects of Markov processes, as they typically arise
from scientific models in the form of stochastic differential and/or
difference equations, are covered in the second part.  The third and
final part describes the filtering approach to estimation of model
states by assimilation of observational data into scientific
models. While most of the material is of survey type, very recent
advances in the field of nonlinear data assimilation covered in
this paper include a discussion of Bayesian inference in the context
of optimal transportation and coupling of random variables, as well as
a discussion of recent advances in ensemble transform filters.
References and sources for further reading material will be listed at
the end of each section.
\end{abstract}

% The file chapterXY.tex contains Chapter XY.

%%%%%%%%%%%%%%%%%%%%%%%%%%%%%%%%%%%

\section{Introduction to Bayesian statistics}

In this section, we summarize the Bayesian approach to statistical inference and
estimation, in which probability is interpreted as a measure of
uncertainty (of the system state, for example). 
Contrary to closely related inverse problem formulations, all variables involved are
considered to be uncertain, and are described as random
variables. Furthermore uncertainty is only discussed in the context of available
information, requiring the computation of conditional probabilities;
Bayes' formula is used for statistical inference. We start with a
short introduction to random variables.

\subsection{Preliminaries}

We start with a \emph{sample space} $\Omega$ which characterizes all
possible outcomes of an experiment. An \emph{event} is a subset of
$\Omega$ and we assume that the set ${\cal F}$ of all events forms a
\emph{$\sigma$-algebra} (\emph{i.e.}, $\mathcal{F}$ is non-empty, and
closed over complementation and countable unions). For example,
suppose that $\Omega = \mathbb{R}$. Then events can be defined by
taking all possible countable unions and complements of intervals
$(a,b] \subset \mathbb{R}$; these are known as the \emph{Borel sets}.

\begin{Definition*}[Probability measure] A probability measure is a function
  $\mathbb{P}:{\cal F} \to [0,1]$ with the following properties:
\begin{itemize}
\item[(i)] \emph{Total probability equals one:} $\mathbb{P}(\Omega) = 1$.
\item[(ii)] \emph{Probability is additive for independent events:} If
  ${A_1,A_2,\ldots,A_n,\ldots}$ is a finite or countable collection of
  events $A_i \in {\cal F}$ and $A_i \cap A_j = \emptyset$ for $i\not=
  j$, then
\[
\mathbb{P}(\cup_i A_i) = \sum_i \mathbb{P}(A_i)
\]
\end{itemize}
The triple $(\Omega,{\cal F},\mathbb{P})$ is called a \emph{probability
space}.
\end{Definition*}

\begin{Definition*}[Random variable]
A function $X:\Omega \to \mathbb{R}$ is called a (univariate) \emph{random
variable} if 
\[
\{\omega\in\Omega : X(\omega) \le x\} \in {\cal F}
\]
for all $x \in \mathbb{R}$. The (cumulative) \emph{probability distribution
function} of $X$ is given by
\[
F_X(x) = \mathbb{P}(\{\omega \in \Omega: X(\omega) \le x\}).
\]
The cumulative probability distribution function implies a probability
measure on $\mathbb{R}$ which we denote by $\mu_X$.
\end{Definition*}

Often, when working with a random variable $X$, the underlying
probability space $(\Omega,{\cal F},\mathbb{P})$ is not emphasised;
one typically only specifies the \emph{target space} ${\cal X} =
\mathbb{R}$ and the probability distribution or \emph{measure} $\mu_X$
on ${\cal X}$. We then say that $\mu_X$ is the \emph{law} of $X$ and write
$X\sim \mu_X$. A probability measure $\mu_X$ introduces an integral
over ${\cal X}$ and
\[
\mathbb{E}_X[f] = \int_{\cal X} f(x) \mu_X({\rm d}x)
\]
is called the \emph{expectation value} of a function
$f:\mathbb{R}\to\mathbb{R} $ ($f$ is called a \emph{measurable
  function} where the integral exists). We also use the notation
$\mbox{law}(X) = \mu_X$ to indicate that $\mu_X$ is the probability
measure for a random variable $X$. Two important choices for $f$ are
$f(x) = x$, which leads to the mean $\bar x = \mathbb{E}_X[x]$ of $X$, and
$f(x) = (x-\bar x)^2$, which leads to the variance $\sigma^2 =
\mathbb{E}_X[(x-\bar x)^2]$ of $X$.

Univariate random variables naturally extend to the multivariate case,
\emph{i.e.}~${\cal X} = \mathbb{R}^N$, $N>1$. A probability measure
$\mu_X$ on ${\cal X}$ is called \emph{absolutely continuous} (with
respect to the standard Lebesgue integral ${\rm d}x$ on $\mathbb{R}^N$)
if there exists a \emph{probability density function} (PDF)
$\pi_X:{\cal X}\to \mathbb{R}$ with $\pi_X(x) \ge 0$, and
\[
\mathbb{E}_X[f] = \int_{\cal X} f(x) \mu_X({\rm d}x) = \int_{\mathbb{R}^N} f(x) \pi_X(x) {\rm d} x,
\]
for all measurable functions $f$. The shorthand $\mu_X({\rm d} x) =
\pi_X {\rm d}x$ is often adopted.  The implication is that
one can, for all practical purposes, work within the classical Riemann
integral framework and does not need to resort to Lebesgue integration. 
Again we can define the mean $\bar x \in
\mathbb{R}^N$ of a multivariate random variable and its covariance matrix
\[
P = \mathbb{E}_X [ (x-\bar x)(x-\bar x)^{\rm T}] \in \mathbb{R}^{N\times N} .
\]
Here $a^{\rm T}$ denotes the transpose of a vector $a$.
We now discuss a few standard distributions.

\begin{Example*}[Gaussian distribution]
  We use the notation $X\sim {\rm N}(m,\sigma^2)$ to denote a
  univariate Gaussian random variable with mean $\bar x$ and variance
  $\sigma^2$, with PDF given by
\[
\pi_X (x) = \frac{1}{\sqrt{2\pi}\sigma}e^{-\frac{1}{2\sigma^2}(x-\bar x)^2} ,
\]
$x \in \mathbb{R}$. In the multivariate case, we use the notation
$X\sim {\rm N}(\bar x,\Sigma)$ to denote a Gaussian random variable with
PDF given by
\[
\pi_X(x) = \frac{1}{(2\pi)^{N/2}|\Sigma|^{1/2}}\exp \left( -\frac{1}{2}
(x-\bar x)^{\rm T} \Sigma^{-1} (x-\bar x) \right) ,
\]
$x \in \mathbb{R}^N$. 
\end{Example*}

\begin{Example*}[Laplace distribution and Gaussian mixtures]
  The univariate Laplace distribution has PDF
\[
\pi_X (x) = \frac{\lambda}{2} e^{-\lambda |x|},
\]
$x\in \mathbb{R}$. This may be rewritten as 
\[
\pi_X(x) = \int_0^\infty
\frac{1}{\sqrt{2\pi}\sigma} e^{-x^2/(2\sigma^2)}
\frac{\lambda^2}{2} e^{-\lambda^2 \sigma/2} {\rm d}\sigma,
\]
which is a weighted Gaussian PDF with mean zero and variance
$\sigma^2$, integrated over $\sigma$. Replacing the integral by a
Riemann sum over a sequence of quadrature points
$\{\sigma_j\}_{j=1}^J$, we obtain
\[
\pi_X(x) \approx  \sum_{j=1}^J \alpha_j \frac{1}{\sqrt{2\pi} \sigma_j}
e^{-x^2/(2\sigma_j^2)}, \qquad \alpha_j 
\propto \frac{\lambda^2}{2}
e^{-\lambda^2 \sigma_j/2} (\sigma_j-\sigma_{j-1})
\]
and the constant of proportionality is chosen such that the weights
$\alpha_j$ sum to one.
This is an example of a \emph{Gaussian mixture} distribution, namely
a weighted sum of Gaussians. In this case, the Gaussians are all 
centred on $x=0$; the most general form of a Gaussian mixture is
\[
\pi_X(x) = \sum_{j=1}^J \alpha_j \frac{1}{\sqrt{2\pi} \sigma_j}
e^{-(x-x_j)^2/(2\sigma_j^2)},
\]
with weights $\alpha_j > 0$ subject to $\sum_{j=1}^J \alpha_j = 1$,
and locations $-\infty< x_j<\infty$. Univariate Gaussian mixtures
generalize to mixtures of multi-variate Gaussians in the obvious
manner.
\end{Example*}

\begin{Example*}[Point distribution]
As a final example, we consider the point measure $\mu_{x_0}$ defined by
\[
\int_{\cal X} f(x) \mu_{x_0}({\rm d} x) =  f(x_0).
\]
Using the Dirac delta notation $\delta (\cdot)$ this can be formally
written as $\mu_{x_0}({\rm d}x) = \delta (x - x_0){\rm d}x$. The associated
random variable $X$ has the certain outcome $X(\omega) = x_0$ for
almost all $\omega \in \Omega$. One can call such a random variable
\emph{deterministic}, and write $X = x_0$ for short.  Note that the
point measure is not absolutely continuous with respect to the
Lebesgue measure, \emph{i.e.}, there is no corresponding probability
density function.
\end{Example*}

We now briefly discuss pairs of random variables $X_1$ and $X_2$ over
the same target space ${\cal X}$. Formally, we can treat them as a
single random variable $Z = (X_1,X_2)$ over ${\cal Z} = {\cal X}
\times {\cal X}$ with a \emph{joint distribution} $\mu_{X_1
  X_2}(x_1,x_2) = \mu_Z(z)$.

\begin{Definition*}[Marginals, independence, conditional probability
  distributions]
 Let $X_1$ and $X_2$ denote two random variables on
  ${\cal X}$ with joint PDF $\pi_{X_1 X_2}(x_1,x_2)$.  The two PDFs
\[
\pi_{X_1}(x_1) = \int_{\cal X} \pi_{X_1 X_2}(x_1,x_2) {\rm d}x_2
\]
and
\[
\pi_{X_2}(x_2) = \int_{\cal X} \pi_{X_1 X_2}(x_1,x_2) {\rm
  d}x_1,
\]
respectively, are called the \emph{marginal PDFs}, \emph{i.e.}
$X_1 \sim \pi_{X_1}$ and $X_2 \sim \pi_{X_2}$. The two random variables are called
\emph{independent} if
\[
\pi_{X_1 X_2}(x_1,x_2) = \pi_{X_1}(x_1) \,\pi_{X_2}(x_2).
\]
We also introduce the \emph{conditional PDFs}
\[
\pi_{X_1}(x_1|x_2) = \frac{\pi_{X_1 X_2}(x_1,x_2)}{\pi_{X_2}(x_2)}
\]
and
\[
\pi_{X_2}(x_2|x_1) = \frac{\pi_{X_1 X_2}(x_1,x_2)}{\pi_{X_1}(x_1)}.
\]
\end{Definition*}

\begin{Example*}[Gaussian joint distributions]
  A Gaussian joint distribution $\pi_{XY}(x,y)$, $x,y\in \mathbb{R}$,
  with mean $(\bar x,\bar y)$ and covariance matrix
\[
\Sigma = \left[ \begin{array}{cc} \sigma_{xx}^2 & \sigma_{xy}^2 \\
    \sigma_{yx}^2 & \sigma_{yy}^2 \end{array} \right]
\]
leads to a Gaussian conditional distribution 
\begin{equation} \label{Gcond1}
\pi_X(x|y) = \frac{1}{\sqrt{2\pi} \sigma_c} e^{-(x-\bar x_c)^2/(2\sigma_c^2)},
\end{equation}
with conditional mean
\[
\bar x_c = \bar x + \sigma_{xy}^2 \sigma_{yy}^{-2} (y - \bar y)
\]
and conditional variance
\[
\sigma_c^2 = \sigma_{xx}^2 - \sigma_{xy}^2 \sigma_{yy}^{-2}
\sigma_{yx}^2.
\]
For given $y$, we define $X|y$ as the random variable with conditional
probability distribution $\pi_X(x|y)$, and write $X|y \sim {\rm
  N}(\bar x_c,\sigma_c^2)$.
\end{Example*}

%%%%%%%%%%%%%%%%%%%%%%%%%%%%%%%%%%%%%%%%%%

\subsection{Bayesian inference} \label{sec:Bayes}

We start this section by considering transformations of random
variables.  A typical scenario is the following one. Given a pair of
independent random variables $\Xi$ with values in ${\cal Y} =
\mathbb{R}^K$ and $X$ with values in ${\cal X} = \mathbb{R}^N$
together with a continuous map $h:\mathbb{R}^N \to \mathbb{R}^K$, we
define a new random variable
\begin{equation} \label{eq:Ymodel}
Y = h(X) + \Xi .
\end{equation}
The map $h$ is called the \emph{observation
  operator}, which yields observed quantities given a particular
value $x$ of the \emph{state variable} $X$, and $\Xi$ represents
\emph{measurement errors}.

\begin{Theorem*}[PDF for transformed random variable]
Assume that both $X$ and $\Xi$ are absolutely continuous, then $Y$ is
absolutely continuous with PDF
\begin{equation}
\label{pi_y}
\pi_Y(y) = \int_{\cal X} \pi_{\Xi}(y-h(x)) \pi_X(x) {\rm d}x .
\end{equation}
If $X$ is a deterministic variable, i.e.~$X = x_0$ for an appropriate
$x_0 \in \mathbb{R}^N$, then the PDF simplifies to
\[
\pi_Y(y) = \pi_{\Xi}(y-h(x_0)) .
\]
\end{Theorem*}

\noindent
\begin{proof}
  We start with $X=x_0$. Then $Y-h(x_0) = \Xi$ which immediately
  implies the stated result. In the general case, consider the 
  conditional probability
  \[
  \pi_Y(y|x_0) = \pi_{\Xi}(y-h(x_0)).
  \]
  Equation \eqref{pi_y} then follows from the implied joint distribution
  \[
  \pi_{XY}(x,y) = \pi_Y(y|x) \pi_X(x)
  \]
and subsequent marginalization, \emph{i.e.}
\[
\pi_Y(y) = 
\int_{\cal X} \pi_{XY}(y,x)  {\rm d}x =
\int_{\cal X} \pi_Y(y|x) \pi_X(x) {\rm d}x .
\]
\end{proof}
The problem of predicting the distribution $\pi_Y$ of $Y$ given a
particular configuration of the state variable $X= x_0$ is called the
\emph{forward problem}. The problem of predicting the distribution 
of the state variable $X$ given an \emph{observation} $Y=y_0$ gives rise
to an \emph{inference problem}, which is defined more formally as follows.

\begin{Definition*}[Bayesian inference]
Given a particular value $y_0 \in \mathbb{R}^K$, we consider the
associated conditional PDF $\pi_X(x|y_0)$ for
the random variable $X$. From
\[
\pi_{XY}(x,y) = \pi_Y(y|x) \pi_X(x) = \pi_X(x|y) \pi_Y(y)
\]
we obtain \emph{Bayes' formula}
\begin{equation} \label{eq:Bayes}
\pi_{X}(x|y_0) = \frac{\pi_X(y_0|x) \pi_X(x)}{\pi_Y(y_0)}
\end{equation}
The object of \emph{Bayesian inference} is to obtain
$\pi_{X}(x|y_0)$.
\end{Definition*}

Since $\pi_Y(y_0) \not= 0$ is a constant, Equation \eqref{eq:Bayes}
can be written as
\[
\pi_{X}(x|y_0) \propto \pi_X(y_0|x) \pi_X(x) = \pi_\Xi (y_0-h(x))
\pi_X(x) ,
\]
where the constant of proportionality depends only on $y_0$.  We
denote $\pi_X(x)$ the \emph{prior PDF} of the random variable $X$ and
$\pi_X(x|y_0)$ the \emph{posterior PDF}. The function $\pi(y_0|x)$ is
called the \emph{likelihood function}.
 
Having obtained a posterior PDF $\pi_X(x|y_0)$, it is often necessary to provide
an estimate of a ``most likely'' value of $x$ conditioned on $y_0$. 
Bayesian estimators for $x$ are defined as follows.

\begin{Definition*}[Bayesian estimators] Given a posterior PDF
  $\pi_X(x|y_0)$ we define a \emph{Bayesian estimator} $\hat x \in
  {\cal X}$ by
\[
\hat x = \mbox{arg min}_{x' \in {\cal X}} \int {\rm L}(x',x) \pi_X(x|y_0){\rm d}x
\]
where ${\rm L} (x',x)$ is an appropriate loss function. Popular
choices include the \emph{maximum a posteriori (MAP) estimator} with
$\hat x$ corresponding to the modal value of $\pi_X(x|y_0)$. The MAP
estimator formally corresponds to the loss function ${\rm L}(x',x) =
{\rm 1}_{\{x'\not = x\}}$. The \emph{posterior median estimator}
corresponds to ${\rm L}(x',x) = \|x'-x\|$ while the \emph{minimum mean
  square error estimator} (or \emph{conditional mean estimator})
\[
\hat x = \int_{\cal X} x\pi_X(x|y_0) {\rm d}x 
\]
 results from ${\rm L}(x',x) = \|x'-x\|^2$.
\end{Definition*}

We now consider an important example for which the posterior can be
computed analytically.

\begin{Example*}[Bayes' formula for Gaussian distributions]
  Consider the case of a scalar observation, \emph{i.e.}~$K=1$, 
  with $\Xi\sim N(0,\sigma_{rr}^2)$.
  Then
  \[
  \pi_\Xi(h(x)-y) =  \frac{1}{\sqrt{2\pi} \sigma_{rr}}
  e^{-\frac{1}{2\sigma_{rr}^2}\left(h(x)-y\right)^2} .
  \]
  We also assume that $X\sim {\rm N}(\bar x,P)$ and that $h(x) =
  Hx$. Then the posterior distribution of $X$ given $y=y_0$ 
  is also Gaussian with mean
  \[
  \bar x_c = \bar x - PH^{\rm T}(HPH^{\rm T} + \sigma_{rr}^2)^{-1}(H\bar x - y_0)
  \]
  and covariance matrix
  \[
  P_c = P - PH^{\rm T}(HPH^{\rm T} + \sigma_{rr}^2)^{-1} HP.
  \]
  These are the famous Kalman update formulas which follow from the
  fact that the product of two Gaussian distributions is also
  Gaussian, where the variance of $Y = HX + \Sigma$ is given by
  \[
  \sigma_{yy}^2 = HPH^{\rm T} + \sigma_{rr}^2 
  \]
  and the vector of covariances between $x\in \mathbb{R}^N$ and $y=Hx
  \in \mathbb{R}$ is given by $PH^{\rm T}$. For Gaussian random
  variables, the MAP, posterior median, and minimum mean square error
  estimators coincide and are given by $\bar x_c$. The case of
  vector-valued observations will be discussed in Section
  \ref{s:enkf}. Finally note that $\bar x_c$ solves the minimization
  problem
\[
\bar x_c = \arg \min_{x\in \mathbb{R}^N} \left\{ \frac{1}{2} (x-\bar
  x)^{\rm T}
  P^{-1} (x-\bar x) + \frac{1}{2R}(Hx - y_0)^2 \right\} ,
\]
which can be viewed as a regularization  of the ill-posed inverse problem
\[
y_0 = Hx, \quad x\in \mathbb{R}^N, \qquad N>1,
\]
in the sense of Tikhonov. A standard Tikhonov regularization would be based on
$P^{-1} = \delta I$ with the regularization parameter $\delta > 0$
appropriately chosen. In the Bayesian approach to inverse problems the
regularization term is instead determined by the Gaussian prior $\pi_X$.
\end{Example*}

We mention in passing that Bayes' formula has to be replaced by the
Radon-Nikodym derivative in the case where the prior distribution is
not absolutely continuous with respect to the Lebegue measure (or in
case the space ${\cal X}$ does not admit a Lebesgue measure). Consider
as an example the case of an empirical measure $\mu_X$ centered about
the $M$ samples $x_i \in {\cal X}$, $i=1,\ldots,M$, \emph{i.e.}
a weighted sum of point measures given by
\[
\mu_X({\rm d}x) = \frac{1}{M} \sum_{i=1}^M \mu_{x_i}({\rm d}x) .
\]
Then the resulting posterior measure $\mu_X(\cdot |y_{\rm obs})$ is absolutely continuous with respect to
$\mu_X$, \emph{i.e.} there exists a \emph{Radon-Nikodym derivative}
such that
\[
\int_{\mathcal{X}} f(x)\mu_X({\rm d} x|y_0)
= \int_{\mathcal{X}} f(x)\frac{{\rm d}\mu_X(x|y_0)}{{\rm d}\mu_X(x)}
\mu_X({\rm d}x)
\]
and the Radon-Nikodym derivative satisfies
\[
\frac{{\rm d}\mu_X(x|y_0)}{{\rm d}\mu_X(x)} \propto \pi_\Xi (h(x)-y_0) .
\]
Furthermore, the explicit expression for the posterior measure is given by
\[
\mu_X({\rm d}x|y_0) = \sum_{i=1}^M w_i\, \mu_{x_i}({\rm d}x),
\]
with weights $w_i\ge 0$ defined by
\[
w_i \propto  \pi_\Xi (h(x_i)-y_0),
\]
and the constant of proportionality is determined by the condition
$\sum_{i=1}^M w_i = 1$.

%%%%%%%%%%%%%%%%%%%%%%%%%%%%%%%%%%%%%%%%%%%%%%%%%%%%%%%

\subsection{Coupling of random variables} \label{sec:OT}

We have seen that under Bayes' formula a prior
probability measure $\mu_X (\cdot)$ on ${\cal X}$ is tranformed into a posterior
probability measure $\mu_X(\cdot|y_0)$ on ${\cal X}$ conditioned on
the observation $y_0 = Y(\omega)$. With each of the probability measures, we can
associate random variables such that, \emph{e.g.}, $X_1 \sim \mu_X$
and $X_2 \sim \mu_X(\cdot|y_0)$. However, while Bayes'
formula leads to a transformation of measures, it does not imply a
specific transformation on the level of the associated random
variables; many different transformations of random variables lead to
the same probability measure. In this section we will, therefore,
introduce the concept of coupling two probability measures.

\begin{Definition*}[Coupling]
  Let $\mu_{X_1}$ and $\mu_{X_2}$ denote two probability measures on a
  space ${\cal X}$. 
% Coupling $\mu_{X_1}$ and $\mu_{X_2}$ means
%   to construct a joint measure $\mu_Z$ on the product space ${\cal Z}
%   = {\cal X} \times {\cal X}$ such that the marginal measure in ${\cal
%     X}_1$ is $\mu_{X_1}$ and the marginal measure in ${\cal X}$ is
%   $\mu_{X_2}$.  
  A \emph{coupling} of $\mu_{X_1}$ and $\mu_{X_2}$ consists of a pair
  $Z = (X_1,X_2)$ of random variables such that $X_1 \sim
  \mu_{X_1}$, $X_2 \sim \mu_{X_2}$, and $Z \sim
  \mu_Z$. The joint measure $\mu_Z$ on the product space ${\cal Z} =
  {\cal X} \times {\cal X}$, is called the {\it transference plan} for
  this coupling. The set of all transference plans is denoted by
  $\Pi(\mu_{X_1},\mu_{X_2})$.
\end{Definition*}

Here, we will discuss different forms of couplings assuming that both
the source and target distributions are explicitly known, whilst applications
to Bayes formula (\ref{eq:Bayes}) will be discussed in Sections
\ref{sec:MC} and \ref{sec:4}. In practice, the source distribution needs often to be estimated
from available realizations of the underlying random variable $X_1$. This is
the subject of parametric and non-parametric statistics and will not
be discussed in this survey paper. In the context of Bayesian
statistics, knowlege of the source (prior) distribution and the
likelihood implies knowledge of the target (posterior) distribution. 

Since prior distributions in Bayesian inference are generally assumed
to be absolutely continuous, the discussion of couplings will be restricted to the
less abstract case of ${\cal X} = \mathbb{R}^N$ and $\mu_{X_1}({\rm d}
x) = \pi_{X_1}(x){\rm d}x$, $\mu_{X_2} ({\rm d} x) =
\pi_{X_2}(x){\rm d}x$. In other words, we assume that the marginal
measures are absolutely continuous.  We will, in general, not assume
that the coupling is absolutely continuous on ${\cal Z} = {\cal X}
\times {\cal X} = \mathbb{R}^{2N}$. Clearly, couplings always exist
since one can use the trivial product coupling
\[
\pi_Z(x_1,x_2) = \pi_{X_1}(x_1) \pi_{X_2}(x_2),
\]
in which case the associated random variables $X_1$ and $X_2$ are
independent. The more interesting case is that of a deterministic
coupling.

\begin{Definition*}[Deterministic coupling] Assume that we have
  a random variable $X_1$ with law $\mu_{X_1}$ and a second
  probability measure $\mu_{X_2}$.  A diffeomorphism $T:{\cal X} \to
  {\cal X}$ is called a \emph{transport map} if the induced random
  variable $X_2= T(X_1)$ satisfies
\[
\int_{{\cal X}} f(x_2) \mu_{X_2}({\rm d}x_2) = \int_{{\cal X}} f(T(x_1))
\mu_{X_1}({\rm d}x_1) 
\]
for all suitable functions $f:{\cal X} \to \mathbb{R}$. The associated coupling 
\[
\mu_Z({\rm d}x_1,{\rm d}x_2) = \delta (x_2 - T(x_1)) \mu_{X_1}({\rm d} x_1) {\rm d}x_2,
\]
where $\delta(\cdot)$ is the standard Dirac distribution,
is called a \emph{deterministic coupling}. Note that $\mu_Z$ is not absolutely continuous 
even if both $\mu_{X_1}$ and $\mu_{X_2}$ are.
\end{Definition*}

Using
\[
\int_{{\cal X} } f(x_2)  \delta (x_2 - T(x_1)) {\rm d}x_2
= f(T(x_1)) ,
\]
it indeed follows from the above definition of $\mu_Z$ that
\[
\int_{\cal X} f(x_2) \mu_{X_2}({\rm d}x_2) = \int_{\cal Z} f(x_2) \mu_Z({\rm d}x_1,{\rm d}x_2) = 
\int_{\cal X} f(T(x_1)) \mu_{X_1}({\rm d}x_1) .
\]
We discuss a simple example.

\begin{Example*}[One-dimensional transport map] Let $\pi_{X_1}(x)\ge 0$ and
  $\pi_{X_2}(x)>0$ denote two PDFs on ${\cal X} = \mathbb{R}$. We define
  the associated \emph{cumulative distribution functions} by
\[
F_{X_1}(x) = \int_{-\infty}^x \pi_{X_1}(x'){\rm d}x', \qquad
F_{X_2}(x) = \int_{-\infty}^x \pi_{X_2}(x'){\rm d}x'.
\]
Since $F_{X_2}$ is monotonically increasing, it has
a unique inverse $F_{X_2}^{-1}(p)$ for $p \in [0,1]$. The inverse may be used 
to define a transport map that transforms $X_1$ into $X_2$ as follows,
\[
X_2 = T(X_1) = F_{X_2}^{-1} (F_{X_1}(X_1)).
\]
For example, consider the case where $X_1$ is a random
variable with uniform distribution ${\rm U}([0,1])$ and $X_2$ is a
random variable with standard normal distribution ${\rm N}(0,1)$.
Then the transport map between $X_1$ and $X_2$ is simply the inverse of
the cumulative distribution function
\[
F_{X_2}(x) = \frac{1}{\sqrt{2\pi}}\int_{-\infty}^x e^{-(x')^2/2} {\rm d}x',
\]
which provides a standard tool for converting uniformly distributed
random numbers to normally distributed ones.
\end{Example*}

We now extend this transform method to random variables in $\mathbb{R}^N$ with
$N=2$. 

\begin{Example*}[Knothe-Rosenblatt rearrangement]
Let $\pi_{X_1}(x^1,x^2)$ and $\pi_{X_2}(x^1,x^2)$ denote two PDFs on
$x = (x^1,x^2) \in \mathbb{R}^2$. A transport map between $\pi_{X_1}$ and $\pi_{X_2}$ can be
constructed in the following manner. We first find the two
one-dimensional marginals $\pi_{X_1^1}(x^1)$ and $\pi_{X_2^1}(x^1)$ of 
the two PDFs. In the previous example we have seen how to construct a
transport map $X_2^1 = T_1(X_1^1)$ which couples these two one-dimensional
marginal PDFs. Here $X^1_i$ denotes the first component of the random
variables $X_i$, $i=1,2$. Next we write
\[
\pi_{X_1}(x^1,x^2) = \pi_{X_1}(x^2|x^1) \pi_{X_1^1}(x^1), \quad
\pi_{X_2}(x^1,x^2) = \pi_{X_2}(x^2|x^1) \pi_{X_2^1}(x^1)
\]
and find a transport map $X_2^2 = T_2(X_1^1,X^2_1)$ by considering 
one-dimensional couplings between $\pi_{X_1}(x^2|x^1)$ and
$\pi_{X_2}(x^2|T(x^1))$ with $x^1$ fixed. The associated
joint distribution is given by
\[
\pi_Z(x_1^1,x_1^2,x_2^1,x_2^2) = \delta (x_2^1-T_1(x_1^1)) \delta (x_2^2-T_2(x_1^1,x_1^2))
\pi_{X_1}(x_1^1,x_1^2).
\]
 \end{Example*}

 This is called the \emph{Knothe-Rosenblatt rearrangement}, also
 well-known to statisticians under the name of \emph{conditional quantile
 transforms}. It can be extended to $\mathbb{R}^N$, $N\ge 3$ in the
 obvious way by introducing the conditional PDFs
\[
\pi_{X_1}(x^3|x^1,x^2), \qquad \pi_{X_2}(x^3|x^1,x^2),
\]
and by constructing an appropriate map $X_2^3 =
T_3(X_1^1,X_1^2,X_1^3)$ from those conditional PDFs for fixed pairs
$(x_1^1,x_1^2)$ and $(x_2^1,x_2^2) = (T_1(x_1^1),T_2(x_1^1,x_1^2))$
{\it etc.}  While the Knothe-Rosenblatt rearrangement can be used in
quite general situations, it has the undesirable property that the map
depends on the choice of ordering of the variables \emph{i.e.}, in two
dimensions a different map is obtained if one instead first
couples the $x^2$ components.

\begin{Example*}[Affine transport maps for Gaussian distributions]
  Consider two Gaussian distributions ${\rm N}(\bar x_1,\Sigma_1)$ and
  ${\rm N}(\bar x_2,\Sigma_2)$ in $\mathbb{R}^N$ with means $\bar x_1$
  and $\bar x_2$ and covariance matrices $\Sigma_1$ and
  $\Sigma_2$, respectively. We first define the \emph{square root}
  $\Sigma^{1/2}$ of a  symmetric positive definite matrix $\Sigma$ as
  the unique symmetric positiv definite matrix which satisfies $\Sigma^{1/2}
  \Sigma^{1/2} = \Sigma$. Then the affine transformation
\begin{equation} \label{eq:Gcoupling}
x_2 = T(x_1) = \bar x_2 + \Sigma_2^{1/2} \Sigma_1^{-1/2} (x_1-\bar x_1)
\end{equation}
provides a deterministic coupling. Indeed, we find that
\[
(x_2-\bar x_2)^{\rm T}\Sigma_2^{-1} (x_2-\bar x_2) = 
(x_1-\bar x_1)^{\rm T}\Sigma_1^{-1} (x_1-\bar x_1) 
\]
under the suggested coupling. The proposed coupling is, of course, not
unique since 
\[
x_2 = T(x_1) = \bar x_2 + \Sigma_2^{1/2} Q \Sigma_1^{-1/2} (x_1-\bar x_1),
\]
where $Q$ is an orthogonal matrix,
also provides a coupling. We will see in Section \ref{s:enkf} that a
coupling between Gaussian random variables is also at the heart of the
ensemble square root filter formulations of sequential data assimilation.
\end{Example*}

Deterministic couplings can be viewed as a special case of
a \emph{Markov process} $\{X_n\}_{n\in\{1,2\}}$ defined by
\[
\pi_{X_2}(x_2) = \int_{{\cal X}_1} \pi(x_2|x_1) \pi_{X_1}(x_1) {\rm d}x_1,
\]
where $\pi(x_2|x_1)$ denotes an appropriate conditional PDF for the
random variable $X_2$ given $X_1 = x_1$. Indeed, we simply have
\[
\pi(x_2|x_1)  = \delta (x_2-T(x_1))
\]
for deterministic couplings. We will come back to Markov processes in
Section \ref{sec:Markov}.

The trivial coupling $\pi_Z(x_1,x_2) = \pi_{X_1}(x_1) \pi_{X_2}(x_2)$
leads to a zero correlation between the induced random variables $X_1$
and $X_2$ since their covariance is
\[
\mbox{cov} (X_1,X_2) = \mathbb{E}_Z[(x_1-\bar x_1)(x_2-\bar x_2)^{\rm T}]
= \mathbb{E}_Z [x_1x_2^{\rm T}] - \bar x_1 \bar x_2^{\rm T} = 0 ,
\]
where $\bar x_i = \mathbb{E}_{X_i}[x]$. 
A transport map leads instead to the covariance matrix
\[
\mbox{cov} (X_1,X_2) = \mathbb{E}_{Z} [x_1x_2^{\rm T}] - \mathbb{E}_{X_1}[x_1] 
(\mathbb{E}_{X_2}[x_2])^{\rm T}  = \mathbb{E}_{X_1}[x_1 T(x_1)^{\rm T} ] - 
\bar x_1 \bar x_2^{\rm T} ,
\]
which is non-zero in general. If several transport maps exists then
one could choose the one that maximizes the covariance.  Consider,
for example, univariate random variables $X_1$ and $X_2$, then maximising
their covariance for given marginal PDFs has an important
geometric interpretation: it is equivalent to minimizing the mean
square distance between $x_1$ and $T(x_1)=x_2$ given by
\begin{align*}
\mathbb{E}_Z[|x_2-x_1 |^2] &= \mathbb{E}_{X_1}[|x_1|^2] + 
\mathbb{E}_{X_2}[|x_2 |^2] - 2 \mathbb{E}_Z[x_1 x_2] \\
&= \mathbb{E}_{X_1}[|x_1|^2] + 
\mathbb{E}_{X_2}[|x_2|^2] - 2 \mathbb{E}_Z[(x_1-\bar x_1)
(x_2-\bar x_2)] -2 \bar x_1 \bar x_2 \\
&= \mathbb{E}_{X_1}[|x_1|^2] + 
\mathbb{E}_{X_2}[|x_2|^2] -2 \bar x_1 \bar x_2 - 2 \mbox{cov}(X_1,X_2).
\end{align*}
Hence finding a joint measure $\mu_Z$ that minimizes the expectation
of $(x_1-x_2)^2$ simultaneously maximizes the covariance between
$X_1$ and $X_2$. This geometric interpretation leads to the celebrated
Monge-Kantorovitch problem.

\begin{Definition*}[Monge-Kantorovitch problem] A transference plan
  $\mu^\ast_Z \in \Pi(\mu_{X_1},\mu_{X_2})$ is called the solution to
  the \emph{Monge-Kantorovitch problem} with cost function $c(x_1,x_2)
  = \|x_1-x_2\|^2$ if
\begin{equation} \label{eq:MK}
\mu_Z^\ast = \arg \inf_{\mu_Z\in \Pi(\mu_{X_1},\mu_{X_2})} \mathbb{E}_Z [ \|x_1-x_2\|^2].
\end{equation}
The associated function
\[
W(\mu_{X_1},\mu_{X_2}) = \mathbb{E}_Z [ \|x_1-x_2\|^2], \qquad 
\mbox{law}(Z) = \mu^\ast_Z
\]
is called the $L^2$-Wasserstein distance of $\mu_{X_1}$ and $\mu_{X_2}$. 
\end{Definition*}

\begin{Theorem*}[Optimal transference plan] If the measures
  $\mu_{X_i}$, $i=1,2$, are absolutely continuous, then the optimal
  transference plan that solves the Monge-Kantorovitch problem
  corresponds to a deterministic coupling with transfer map
  \[
  X_2 = T(X_1) = \nabla_x\psi(X_1),
  \]
  for some convex potential  $\psi:\mathbb{R}^N \to
  \mathbb{R}$.\end{Theorem*} 

\begin{proof} We only demonstrate that the solution to the
  Monge-Kantorovitch problem is of the
  desired form when the infimum in (\ref{eq:MK}) is restricted to
  deterministic couplings. See \cite{sr:Villani}
  for a complete proof and also for more general results in terms of
  subgradients and weaker conditions on the two marginal measures.

  We denote the associated PDFs by $\pi_{X_i}$, $i=1,2$. We also
  introduce the inverse transfer map $X_1 = S(X_2) = T^{-1}(X_2)$ and
  consider the functional
\begin{align*}
{\cal L}[S,\Psi] &= \frac{1}{2} \int_{\mathbb{R}^N} \|S(x)-x\|^2
\pi_{X_2}(x){\rm d}x\, + \\
& \qquad \int_{\mathbb{R}^N} \left[ \Psi(S(x)) \pi_{X_2}(x) - \Psi(x)
\pi_{X_1}(x) \right] {\rm d}x 
\end{align*}
in $S$ and a potential $\Psi:\mathbb{R}^N \to \mathbb{R}$. 
We note that
\begin{align*}
& \int_{\mathbb{R}^N} \left[ \Psi(S(x)) \pi_{X_2}(x) - \Psi(x)
\pi_{X_1}(x) \right] {\rm d}x = \\
& \qquad \qquad \int_{\mathbb{R}^N} \Psi(x) \left[
\pi_{X_2}(T(x))|DT(x)| - \pi_{X_1}(x) \right] {\rm d}x
\end{align*}
by a simple change of variables. Here
$|DT(x)|$ denotes the determinant of the Jacobian matrix of $T$ at $x$
and the potential $\Psi$ can be
interpreted as a Lagrange multiplier enforcing the coupling of the two
marginal PDFs under the desired transport map. 

Taking variational derivatives with respect to $S$ and $\Psi$, we obtain two equations
\[
\frac{\delta {\cal L}}{\delta S} = \pi_{X_2}(x)\left[(S(x)-x) + 
\nabla_x \Psi(S(x))\right] = 0
\]
and
\begin{equation} \label{eq:Mass}
\frac{\delta {\cal L}}{\delta \Psi} = -\pi_{X_1}(x) + \pi_{X_2}(T(x))
|DT(x)| = 0
\end{equation}
characterizing critical points of the functional ${\cal L}$. 
The first equality implies
\[
x_2= x_1+ \nabla_x \Psi(x_1) = \nabla_x \left( \frac{1}{2} x_1^{\rm T}x_1 +
  \Psi(x_1) \right) =: \nabla_x \psi(x_1) 
\]
and the second recovers our \emph{Ansatz} that $T$ transforms
$\pi_{X_1}$ into $\pi_{X_2}$ as a result of the Lagrange multiplier $\Psi$.
\end{proof}

\begin{Example*}[Optimal transport maps for Gaussian distributions]
  Consider two Gaussian distributions ${\rm N}(\bar x_1,\Sigma_1)$ and
  ${\rm N}(\bar x_2,\Sigma_2)$ in $\mathbb{R}^N$ with means $\bar x_1$
  and $\bar x_2$ and covariance matrices $\Sigma_1$ and
  $\Sigma_2$, respectively. We had previously discussed the
  deterministic coupling (\ref{eq:Gcoupling}). However, the induced
  affine transformation $x_2 = T(x_1)$ cannot not  be generated from a
  potential $\psi$ since the matrix $\Sigma_2^{1/2} \Sigma_1^{-1/2}$
  is not symmetric. Indeed the optimal coupling in the sense of
  Monge-Kantorovitch with cost function $c(x_1,x_2) = \|x_1-x_2\|^2$ is provided by 
\begin{equation} \label{eq:OGcoupling}
x_2 = T(x_1) := \bar x_2 + \Sigma_2^{1/2} \left[ \Sigma_2^{1/2}
  \Sigma_1 \Sigma_2^{1/2}\right]^{-1/2} \Sigma_2^{1/2} \left(x_1 -
  \bar x_1 \right).
\end{equation}
See \cite{sr:olkin82} for a derivation. The following generalization will
be used in  Section \ref{s:enkf}. Assume that a matrix $A\in
\mathbb{R}^{N\times M}$ is given
such that $\Sigma_2 = A A^{\rm T}$. Clearly we can chose $A =
\Sigma_2^{1/2}$ in which case $M=N$ and $A$ is symmetric. 
However we allow for $A$ to be non-symmetric and $M$ can be different from $N$. 
An important observation is that one can replace $\Sigma_2^{1/2}$ in
(\ref{eq:OGcoupling}) by $A$ and $A^{\rm T}$, respectively, \emph{i.e.}
\begin{equation} 
 \label{eq:OGcoupling2}
T(x_1) = \bar x_2 + A \left[ A^{\rm T}
  \Sigma_1 A \right]^{-1/2} A^{\rm T} \left(x_1 -
  \bar x_1 \right).
\end{equation}
\end{Example*}

While optimal couplings are of broad theoretical and practical
interest, their computational implementation can be very demanding. In
Section \ref{sec:4}, we will discuss an embedding method originally
due to J\"urgen Moser \cite{sr:Moser65}, which leads to a generally
non-optimal but computationally more tractable formulation in the context
of Bayesian statistics and data assimilation.

%%%%%%%%%%%%%%%%%%%%%%%%%%%%%%%%%%%%%%%%%%%%%%%%%%%%%%

\subsection{Monte  Carlo methods} \label{sec:MC}

Monte Carlo methods, also called particle or ensemble methods
depending on the context in which they are being used, can be used to
approximate statistics, namely expectation values $\mathbb{E}_X[f]$,
for a random variable $X$.  We begin by discussing the special case
$f(x) = x$, namely, the mean.

\begin{Definition*}[Empirical mean]
  Given a sequence $X_i$, $i=1,\ldots,M$, of independent random
  variables with identical measure $\mu_X$, the \emph{empirical
  mean} is
\[
\bar x_M = \frac{1}{M} \sum_{i=1}^M X_i(\omega) = 
\frac{1}{M} \sum_{i=1}^M x_i
\]
with samples $x_i = X_i(\omega)$. 
\end{Definition*}

Of course, $\bar x_M$ itself is the realization of a random variable $\bar
X_M$ and we consider the \emph{mean squared error} (MSE)
\begin{align} \nonumber 
{\rm MSE}(\bar x) &= \mathbb{E}_{\bar X_M}[(\bar x_M-\bar x)^2] \\ &= 
(\mathbb{E}_{\bar X_M}[\bar x_M]-\bar x)^2 + \mathbb{E}_{\bar
  X_M}\left[(\bar x_M -
\mathbb{E}_{\bar X_M}[\bar x_M])^2 \right] \label{eq:MSE}
\end{align}
with respect to the exact mean value $\bar x = \mathbb{E}_X[x]$.
We have broken down the MSE into two components: squared bias and
variance. Such a decomposition is possible for any estimator and is known as the
\emph{bias-variance decomposition}. The particular estimator $\bar
X_M$ is called \emph{unbiased} since
$\mathbb{E}_{\bar X_M}[\bar x_M] = \bar x$ for any $M>1$. Furthermore
$\bar X_M$ converges weakly to $\bar x$ under the central limit
theorem provided $\mu_X$ has finite second-order
moments, \emph{i.e.}
\[
\lim_{M\to \infty} \mathbb{E}_{\bar X_M} \left[(\bar x_M -
\mathbb{E}_{\bar X_M}[\bar x_M])^2 \right] = 0.
\]

It remains to generate samples $x_i = X_i(\omega)$ from the
required distribution. Methods to do this include the von Neumann
rejection method and Markov chain Monte Carlo methods, which we will
briefly discuss in Section \ref{sec:Markov}. Often the prior
distribution is assumed to be Gaussian, in which case explicit random
number generators are available. We now turn to the situation where
samples from the prior distribution are available, and are to be used
to approximate the mean of the posterior distribution (or any other
expectation value).

Importance sampling is a classical method to approximate expectation
values of a random variable $X^t \sim \pi_{X^t}$ using samples from a random variable
$X^p \sim \pi_{X^p}$, which requires that the target PDF $\pi_{X^t}$ is absolutely
continuous with respect to proposal PDF $\pi_{X^p}$. This is the case
for the prior and posterior PDFs from Bayes' formula (\ref{eq:Bayes}), \emph{i.e.}~we
set the proposal distribution $\pi_{X^p}(x)$ equal to the prior distribution
$\pi_X(x)$ and the posterior distribution $\pi_X(x|y_0) \propto
\pi_Y(y_0|x)\pi_X(x)$ becomes the target distribution $\pi_{X^t}(x)$.
\begin{Definition*}[Importance sampling for Bayesian estimation]
  Let $x_i^{\rm prior}$, $i=1,\ldots,M$, denote samples from the prior
  PDF $\pi_X(x)$, then the \emph{importance sampler} estimate of the mean
  of the posterior $\pi_X(x|y_0)$ is
  \begin{equation}\label{meanposterior}
    \bar x_M^{\rm post} = \sum_{i=1}^M w_i x_i^{\rm prior}
  \end{equation}
  with \emph{importance weights}
\begin{equation} \label{weightsposterior}
w_i = \frac{\pi_Y (y_0|x_i^{\rm prior})}{
\sum_{i=1}^M \pi_Y(y_0|x_i^{\rm prior})}.
\end{equation}
\end{Definition*}

Importance sampling becomes statistically inefficient when the weights
have largely varying magnitude, which becomes particularly significant
for high-dimensional problems. To demonstrate this effect consider a
uniform prior on the unit hypercube $V = [0,1]^N$. Each of the $M$
samples $x_i$ from this prior formally represent a hypercube with
volume $1/M$. However, the likelihood measures the distance of a
sample $x_i$ to the observation $y_0$ in the Euclidean distance and
the volume of a hypersphere decreases rapidly relative to that of an
associated hypercube as $N$ increases. Within the framework of
the bias-variance decomposition of a mean squared error such as 
(\ref{eq:MSE}), the curse of dimensionality manifests itself in large
variances for finite $M$.

To counteract this curse of dimensionality, one may utilize the
concept of coupling. In other words, assume that we have a transport
map $x^{\rm post} = T(x^{\rm prior})$ which couples the prior and posterior
distributions. Then, with transformed samples
$x_i^{\rm post}  = T(x_i^{\rm prior})$, $i=1,\ldots,M$, we obtain the estimator 
\[
\bar x_M^{\rm post} = \sum_{i=1}^M \hat w_i x_i^{\rm post}
\]
with equal weights $\hat w_i=1/M$. 

Sometimes one cannot couple the prior and posterior distribution
directly, or the coupling is too expensive computationally. Then one
can attempt to find a coupling between the prior PDF $\pi_X(x)$ and an
approximation $\tilde \pi_X(x|y_0)$ to the posterior PDF $\pi_X(x|y_0)
\propto \pi_Y(y_0|x) \pi_X(x)$. Given an associated transport map
$X^{\rm prop} = \tilde T(X^{\rm prior})$, \emph{i.e.}
\[
\tilde \pi_X(\tilde T(x)|y_0) = \pi_X(x) |D\tilde T(x)|^{-1},
\]
one then takes $\tilde \pi_X(x|y_0)$ as the proposal density $\pi_{X^p}(x)$ in an
importance sampler with realizations $x_i^{\rm prop}$, $i=1,\ldots,M$, defined by 
\[
x_i^{\rm prop} = \tilde T(x^{\rm prior}_i).
\]
An asymptotically unbiased estimator for the posterior mean is now provided by
\begin{equation} \label{eq:importantmean}
\bar x_M^{\rm post} = \sum_{i=1}^M \tilde w_i x_i^{\rm prop}
\end{equation}
with  weights
\begin{equation} \label{eq:importantweights}
\tilde w_i \propto \frac{\pi_Y (y_0|x_i^{\rm prop})
  \pi_X(x_i^{\rm prop})}{\tilde \pi_X(x_i^{\rm prop}|y_0)} =
\pi_Y(y_0|x_i^{\rm prop}) |D\tilde T(x_i^{\rm prior})| \frac{\pi_X(x_i^{\rm
    prop})}{\pi_X(x_i^{\rm prior})},
\end{equation}
$i=1,\ldots,M$. The constant of proportionality is chosen such that
$\sum_{i=1}^M \tilde w_i = 1$.  Indeed, if $\pi_{X^p}(x) = \tilde
\pi_X(x|y_0) = \pi_X(x|y_0)$, we recover the case of equal weights
$\tilde w_i = 1/M$, and $\pi_{X^p}(x)=\tilde \pi_X(x|y_0) = \pi_X(x)$
leads to standard importance sampling using prior samples,
\emph{i.e.}~$x_i^{\rm prop} = x_i^{\rm prior}$.

We will return to the subject of sampling from the posterior
distribution in Sections \ref{sec:EPM} and \ref{sec:SMCM}. 

%%%%%%%%%%%%%%%%%%%%%%%%%%%%%%%%%%%%%%%%%%%%%%%%%%%%%%%%%%%%%%%%%%%%%

\subsection*{References} An excellent introduction to many topics
covered in this survey is \cite{sr:jazwinski}. Bayesian inference and a Bayesian
perspective on inverse problems are discussed in \cite{sr:kaipio}, \cite{sr:Neal},
\cite{sr:Lewis}. The monographs  \cite{sr:Villani,sr:Villani2} provide
an in depth introduction to optimal transportation and
coupling of random variables. Monte Carlo methods are covered in
\cite{sr:Liu}. We also point to \cite{sr:hastie} for a discussion of
estimation and regression methods from a bias-variance perspective. 
A discussion of infinite-dimensional Bayesian
inference problems can be found in \cite{sr:stuart10a}.

%%%%%%%%%%%%%%%%%%%%%%%%%%%%%%%%%%%%%%%%%%%%%%%%%%%%%%%%%%%%%%%%%%%%%%
%%%%%%%%%%%%%%%%%%%%%%%%%%%%%%%%%%%%%%%%%%%%%%%%%%%%%%%%%%%%%%%%%%%%%%%

\section{Elementary stochastic processes} \label{sec:Markov}

In this section, we collect basic results concerning stochastic
processes which are of relevance for the data assimilation problem.

\begin{Definition*}[Stochastic process] Let $T$ be a set of indices. A
  \emph{stochastic process} is a family $\{X_t\}_{t\in T}$ of random
  variables on a common space ${\cal X}$, \emph{i.e.}~$X_t(\omega) \in {\cal
    X}$.
\end{Definition*}

In the context of dynamical systems, the variable $t$ corresponds to
time. We distinguish between continuous time $t \in [0,t_{\rm end}]
\subset \mathbb{R}$ or discrete time $t_n = n\Delta t$, $n\in
\{0,1,2,\ldots\} = T$, with $\Delta t>0$ a time-increment. In cases
where subscript indices can be confusing we will also use the
notations $X(t)$ and $X(t_n)$, respectively.

A stochastic process can be seen as a function of two arguments: $t$
and $\omega$. For fixed $\omega$, $X_t(\omega)$ becomes a function of
$t\in T$, which we call a realization or trajectory of the stochastic
process. We will restrict to the case where $X_t(\omega)$ is
continuous in $t$ (with probability 1) in the case of a continuous
time. Alternatively, one can fix the time $t\in T$ and consider the
random variable $X_t(\cdot)$ and its distribution. More generally, one
can consider $l$-tuples $(t_1,t_2,\ldots,t_l)$ and associated
$l$-tuples of random variables
$(X_{t_1}(\cdot),X_{t_2}(\cdot),\ldots,X_{t_l}(\cdot))$ and their
joint distributions. This leads to concepts such as temporal
correlation.

\subsection{Discrete time Markov processes}

First, we develop the concept of Markov processes for discrete time
processes.

\begin{Definition*}[Discrete time Markov processes] The discrete time
  stochastic process $\{X_n\}_{n\in T}$ with ${\cal X} = \mathbb{R}^N$
  and $T = \{0,1,2,\ldots)$ is called a (time-independent) \emph{Markov
  process} if its joint PDFs can be written as
\[
\pi_n(x_0,x_1,\ldots,x_n) = \pi(x_n|x_{n-1}) \pi(x_{n-1}|x_{n-2}) \cdots \pi(x_1|x_0)
\pi_0(x_0)
\]
for all $n\in \{0,1,2,\ldots\} = T$. The associated marginal distributions $\pi_n = \pi_{X_n}$ 
satisfy the \emph{Chapman-Kolmogorov equation}
\begin{equation} \label{eq:CK}
\pi_{n+1}(x') = \int_{\mathbb{R}^N} \pi(x'|x) \pi_n(x) {\rm
  d}x 
\end{equation}
and the process can be recursively repeated to yield a family of
marginal distributions $\{\pi_n\}_{n\in T}$ for given $\pi_0$. This
family can also be characterized by the linear \emph{Frobenius-Perron operator}
\begin{equation} \label{eq:Frobenius}
\pi_{n+1} = {\cal P} \pi_n,
\end{equation}
which is induced by (\ref{eq:CK}).
\end{Definition*}

The above definition is equivalent to the more traditional definition
that a process is Markov if the conditional distributions satisfy
\[
\pi_n(x_n|x_0,x_1,\ldots,x_{n-1}) = \pi (x_n|x_{n-1}).
\]

Note that, contrary to
Bayes' formula (\ref{eq:Bayes}), which directly yields marginal
distributions, the Chapman-Kolmogorov equation (\ref{eq:CK}) starts
from a given coupling 
\[
\pi_{X_{n+1} X_n}(x_{n+1},x_n) = \pi (x_{n+1}|x_n)
  \pi_{X_n}(x_n)
\]
followed by marginalization to derive $\pi_{X_{n+1}}(x_{n+1})$. A
Markov process is called time-dependent if the conditional PDF
$\pi(x'|x)$ depends on $t_n$. While we have considered
time-independent processes in this section, we will see in
Section \ref{sec:4} that the idea of coupling applied to Bayes'
formula leads to time-dependent Markov processes.

%%%%%%%%%%%%%%%%%%%%%%%%%%%%%%%%%%%%%%%%%%%%%%%

\subsection{Stochastic difference and differential equations} \label{sec:SDEs}
\label{s:sdes}
We start from the \emph{stochastic difference equation}
\begin{equation} \label{eq:EMmethod}
X_{n+1} = X_n + \Delta t f(X_n) + \sqrt{2\Delta t} Z_n,
\quad t_{n+1} = t_n + \Delta t,
\end{equation}
where $\Delta t >0$ is a small parameter (the step-size), $f$ is a
given (Lipschitz continuous) function, and $Z_n \sim {\rm N}(0,Q)$ are
independent and identically distributed random variables with correlation matrix $Q$.

The time evolution of the associated marginal densities $\pi_{X_n}$ is
governed by the \emph{Chapman-Kolmogorov equation} with conditional PDF
\begin{align} \nonumber
\pi(x'|x) &= \frac{1}{(4\pi \Delta t)^{N/2}  |Q|^{1/2}} \times \\
& \qquad \quad \exp
\left( -\frac{1}{4\Delta t}
(x' - x - \Delta t f(x))^{\rm T} Q^{-1} (x'-x-\Delta t  f(x)) \right). \label{eq:CKFP}
\end{align}
 
\begin{Proposition*}[Stochastic differential and Fokker-Planck equation]
Taking the limit $\Delta t \to 0$, one obtains the stochastic
differential equation (SDE)
\begin{equation} \label{eq:SDE}
{\rm d}X_t = f(X_t){\rm d}t + \sqrt{2} Q^{1/2} {\rm d}W_t
\end{equation}
for $X_t$, where $\{W_t\}_{t\ge 0}$ denotes standard $N$-dimensional Brownian
motion, and the Fokker-Planck equation
\begin{equation} \label{eq:FokkerPlanck}
\frac{\partial \pi_{X}}{\partial t} = -\nabla_x \cdot (\pi_{X} f)
+ \nabla_x \cdot (Q \nabla_x \pi_{X} )
\end{equation}
for the marginal density $\pi_X(x,t)$.  Note that $Q=0$ (no noise) leads to
the Liouville, transport or continuity equation
\begin{equation} \label{eq:Liouville}
\frac{\partial \pi_{X}}{\partial t} = -\nabla_x \cdot (\pi_{X} f),
\end{equation}
which implies that we may interpret $f$ as a given velocity field in
the sense of fluid mechanics.
\end{Proposition*}

\begin{proof}  The difference equation (\ref{eq:EMmethod}) is called
the Euler-Maruyama method for approximating the SDE (\ref{eq:SDE}). 
See \cite{sr:Higham,sr:Kloeden} for a discussion on the convergence
of (\ref{eq:EMmethod}) to (\ref{eq:SDE}) at $\Delta t \to 0$.

The Fokker-Planck equation (\ref{eq:FokkerPlanck}) is the linear
combination of a drift and a diffusion term. To simplify the
discussion we derive both terms separately from (\ref{eq:EMmethod}) by
first considering $f=0$, $Q\not= 0$ and then $Q=0$, $f\not= 0$. 
To simplify the derivation of the diffusion term even further we also assume
$x\in\mathbb{R}$ and $Q=1$. In other words, we show that scalar Brownian motion
\[
{\rm d}X_t = \sqrt{2} {\rm d}W_t
\]
leads to the heat equation
\[
\frac{\partial \pi_X}{\partial t} = \frac{\partial^2
  \pi_X}{\partial x^2}.
\]
We first note that the conditional PDF (\ref{eq:CKFP}) reduces to
\begin{align*} 
\pi(x'|x) &= (4\pi\Delta t)^{-1/2} \exp
\left( -\frac{(x'-x)^2}{4\Delta t} \right)
\end{align*}
under $f(x) = 0$, $Q=1$, $N=1$, and the
Chapman-Kolmogorov equation (\ref{eq:CK}) becomes 
\begin{equation} \label{eq:h1}
\pi_{n+1}(x') = \int_{\mathbb{R}} \frac{1}{\sqrt{4\pi \Delta t}}
e^{-y^2/(4\Delta t)}
\pi_n(x'+y){\rm d}y 
\end{equation}
under the variable substitution $y = x-x'$.
We now expand $\pi_n(x'+y)$ in $y$ about $y=0$, \emph{i.e.}~
\[
\pi_n(x'+y) = \pi_n(x') + y\frac{\partial \pi_n}{\partial x}(x') +
\frac{y^2}{2} \frac{\partial^2 \pi_n}{\partial x^2}(x') + \cdots ,
\]
and substitute the expansion into (\ref{eq:h1}):
\begin{align*}
\pi_{n+1}(x') &=  \int_{\mathbb{R}} \frac{1}{\sqrt{4\pi \Delta t}}
e^{-y^2/(4\Delta t}
\pi_n(x') {\rm d}y \\
& \qquad +  \int_{\mathbb{R}} \frac{1}{\sqrt{4\pi \Delta t}}
e^{-y^2/(4\Delta t)}
y \frac{\partial \pi_n}{\partial x} (x') {\rm d}y \\
& \qquad  +  \int_{\mathbb{R}} \frac{1}{\sqrt{4\pi \Delta t}} e^{-y^2/(4\Delta t)}
\frac{y^2}{2} \frac{\partial^2 \pi_n}{\partial x^2} (x') {\rm d}y +
\cdots .
\end{align*}
The integrals correspond to the zeroth, first and second-order moments of the
Gaussian distribution with mean zero and variance $2 \Delta t$. Hence
\[
\pi_{n+1}(x') = \pi_n(x') + 
\Delta t \frac{\partial^2 \pi_n}{\partial x^2} (x') + \cdots 
\]
and it can also easily be shown that the neglected higher-order terms
contribute with ${\cal O}(\Delta t^2)$ terms. Therefore
\[
\frac{\pi_{n+1}(x')-\pi_n(x_n)}{\Delta t} = 
\frac{\partial^2 \pi_n}{\partial x^2} (x') + {\cal O}(\Delta t),
\]
and the heat equation is obtained upon taking the limit $\Delta t\to
0$. The non-vanishing drift case, \emph{i.e.}~$f(x)\not=0$, while being more
technical, can be treated in the same manner.

One can also use (\ref{eq:Mass}) to derive Liouville's equation
(\ref{eq:Liouville}) directly. We set
\[
T(x) = x + \Delta t f(x)
\]
and note that
\[
|DT(x)| = 1 + \Delta t \nabla_x \cdot f + {\cal O}(\Delta t^2).
\]
Hence (\ref{eq:Mass}) implies
\[
\pi_{X_1} = \pi_{X_2} + \Delta t \pi_{X_2} \nabla_x \cdot f + \Delta t (\nabla_x
\pi_{x_2} ) \cdot f + {\cal O}(\Delta t^2)
\]
and
\[
\frac{\pi_{X_2}-\pi_{X_1}}{\Delta t} = -\nabla_x \cdot (\pi_{X_2}
f)  + {\cal O}(\Delta t).
\]
Taking the limit $\Delta t \to 0$, we obtain (\ref{eq:Liouville}). 
\end{proof}

Following the work of Felix Otto (see, \emph{e.g.},
\cite{sr:Otto01,sr:Villani}), we note that in the case of pure
diffusion, \emph{i.e.} $f=0$, the Fokker-Planck equation can be
rewritten as a gradient flow system.  We first introduce some
notation.

\begin{Definition*}[differential geometric structure on manifold of
  probability densities]
We formally introduce the \emph{manifold
of all PDFs on} ${\cal X} = \mathbb{R}^N$ 
\[
{\cal M} = \{\pi:\mathbb{R}^N\to \mathbb{R}: \pi(x)\ge 0,
\,\int_{\mathbb{R}^N} \pi(x){\rm d}x = 1\}
\]
with \emph{tangent space}
\[
T_\pi{\cal M} = \{ \phi:\mathbb{R}^N\to \mathbb{R}:
\int_{\mathbb{R}^N} \phi(x){\rm d}x = 0\} .
\]
The \emph{variational derivative} of a functional $F:{\cal M}\to
\mathbb{R}$ is defined as
\[
\int_{\mathbb{R}^N} \frac{\delta F}{\delta \pi} \phi \, {\rm d}x = 
\lim_{\epsilon \to 0} \frac{F(\pi + \epsilon \phi)-F(\pi)}{\epsilon}.
\]
where $\phi$ is a function such that $\int_{\mathbb{R}^N} \phi {\rm d}x = 0$, \emph{i.e.}~$\phi \in
T_\pi{\cal M}$.
\end{Definition*}

Consider the potential
\begin{equation} \label{eq:potD}
V(\pi_X) = \int_{\mathbb{R}^n} \pi_X \ln \pi_X  {\rm d}x,
\end{equation}
which has functional derivative 
\[
\frac{\delta V}{\delta \pi_X} = \ln \pi_X,
\]
since 
\begin{align*}
V(\pi_X + \epsilon \phi)  &= V(\pi_X) + \epsilon \int_{\mathbb{R}^N}
( \phi \ln \pi_X + \phi) {\rm d}x + {\cal O}(\epsilon^2)\\
&= V(\pi_X) + \epsilon \int_{\mathbb{R}^N}
\phi \ln \pi_X \,{\rm d}x + {\cal O}(\epsilon^2),
\end{align*}
Hence, we find that the diffusion part of the Fokker-Planck equation
is equivalent to
\begin{equation} \label{eq:Otto}
\frac{\partial \pi_X}{\partial t} = \nabla_x \cdot (Q \nabla_x \pi_X) =
\nabla_x \cdot \left\{ \pi_X Q \nabla_x \frac{\delta V}{\delta \pi_X} \right\}.
\end{equation}
This formulation allows us to treat diffusion in form of a vector field
\[
v(x,t) = -Q \nabla_x \frac{\delta V}{\delta \pi_X} 
\]
which, contrary to vector fields arising from the theory of ordinary
differential equations, depends on the PDF $\pi_X$. See the following
Section \ref{sec:EPM} for an application.

\begin{Proposition*}[Gradient on the manifold of probability densities]
Let $g_{\pi}$ be a metric tensor defined on $T_\pi{\cal M}$ as
\[
g_\pi(\phi_1,\phi_2) = \int_{\mathbb{R}^N}  (\nabla_x \psi_1) \cdot
({\rm M} \nabla_x \psi_2)\,\pi  {\rm d}x
\]
with potentials $\psi_i$, $i=1,2$, determined by the elliptic
partial differential equation (PDE)
\[
-\nabla_x \cdot( \pi {\rm M}\nabla_x \psi_i) = \phi_i,
\]
where ${\rm M} \in \mathbb{R}^{N\times N}$ is a symmetric, positive-definite matrix.

Then the gradient of a potential $F(\pi)$ under $g_\pi$ satisfies
\begin{equation} \label{gradient}
\mbox{\rm grad}_\pi F(\pi) = -\nabla_x \cdot \left( \pi {\rm M}
  \nabla_x \frac{\delta F}{\delta \pi}  \right).
\end{equation}
\end{Proposition*}

\begin{proof} Given the metric tensor $g_\pi$, the gradient is defined 
by
\begin{equation} \label{eq:gradotto}
g_\pi(\mbox{grad}_\pi F(\pi),\phi) = \int_{\mathbb{R}^N} 
 \frac{\delta F}{\delta \pi} \phi {\rm d}x
\end{equation}
for all $\phi \in T_\pi {\cal M}$. Since any element $\phi \in
T_\pi{\cal M}$ can be written in the form 
\[
\phi = -\nabla_x \cdot (\pi {\rm M}\nabla_x \psi)
\]
with suitable potential $\psi$, a potential $\widehat{\psi}$
exists such that 
\[
\mbox{grad}_\pi F(\pi) = -\nabla_x \cdot(\pi {\rm M} \nabla_x \widehat
\psi) \in T_\pi{\cal M}
\]
and we need to demonstrate that
\[
\widehat \psi = \frac{\delta F}{\delta \pi}
\]
is consistent with (\ref{eq:gradotto}). Indeed, we find that
\begin{align*}
\int_{\mathbb{R}^N} \frac{\delta F}{\delta \pi} \phi {\rm d}x &=
-\int_{\mathbb{R}^N} \frac{\delta F}{\delta \pi} \nabla_x \cdot (\pi
{\rm M} \nabla_x \psi ) {\rm d}x \\
&= \int_{\mathbb{R}^N} \pi \nabla_x \frac{\delta F}{\delta \pi} \cdot
({\rm M} \nabla_x \psi) {\rm d} x \\
&= \int_{\mathbb{R}^N} ( \nabla_x \widehat \psi) \cdot ({\rm M}
\nabla_x \psi) \pi {\rm d} x \\ 
&= g_\pi(\mbox{grad}F(\pi), \phi).
\end{align*}
\end{proof}

It follows that the diffusion part of the Fokker-Planck equation can be
viewed as a gradient flow on the manifold ${\cal M}$. More precisely,  
set $F(\pi)  = V(\pi_X)$ and ${\rm M} = Q$ to reformulate (\ref{eq:Otto})
as a gradient flow
\[
\frac{\partial \pi_{X}}{\partial t} = - \mbox{grad}_{\pi_X} V(\pi_X) 
\]
with potential (\ref{eq:potD}).
We will find in Section \ref{sec:4} that related geometric structures
arise from Bayes' formula in the context of filtering. We finally note that
\begin{align*}
\frac{{\rm d}V}{{\rm d}t} &= \int_{\mathbb{R}^N} \frac{\delta V}{\delta \pi_X}  
\frac{\partial \pi_X}{\partial t} {\rm d}x \\
&= -\int_{\mathbb{R}^N}  \left(\nabla_x \frac{\delta V}{\delta
      \pi_X} \right)  
 \cdot  \left({\rm M} \nabla_x \frac{\delta V}{\delta \pi_X}  \right) \pi_X {\rm d}x 
\le 0.
\end{align*}

%%%%%%%%%%%%%%%%%%%%%%%%%%%%%%%%%%%%%%%%%%%

\subsection{Ensemble prediction and sampling methods} \label{sec:EPM}

In this section, we extend the Monte Carlo method from Section
\ref{sec:MC} to the approximation of the marginal PDFs $\pi_X(x,t)$,
$t\ge 0$, evolving under the SDE model (\ref{eq:SDE}). Assume that we have a set of
independent samples $x_i(0)$, $i=1,\ldots,M$, from the initial PDF
$\pi_X(x,0)$.

\begin{Definition*}[ensemble prediction]
A Monte Carlo approximation to
the time-evolved marginal PDFs $\pi_X(x,t)$ can be obtained 
from solving the SDEs
\begin{equation} \label{eq:MCSDE}
{\rm d} x_i = f(x_i){\rm d}t + \sqrt{2}Q^{1/2} {\rm d}W_i(t)
\end{equation}
for $i=1,\ldots,M$, where $\{W_i(t)\}_{i=1}^M$ denote realizations of
independent standard $N$-dimensional Brownian motion and the initial
conditions $\{x_i(0)\}_{i=1}^M$ are realizations of the initial PDF
$\pi_X(x,0)$. This approximation provides an example for a
\emph{particle} or \emph{ensemble prediction} method and it can be
shown that the estimator
\begin{equation}  \label{eq:EME}
\bar x_M(t) = \frac{1}{M} \sum_{i=1}^M x_i(t)
\end{equation}
provides a consistent and unbiased approximation to
the mean $\mathbb{E}_{X_t}[x]$.
\end{Definition*}

Alternatively, using formulation (\ref{eq:Otto}) of the Fokker-Planck
equation (\ref{eq:FokkerPlanck}) in the pure diffusion case, we may
reformulate the random part in (\ref{eq:MCSDE}) and introduce particle
equations
\begin{align} \nonumber
\frac{{\rm d}x_i}{{\rm d}t} &=
f(x_i) - Q \nabla_x \frac{\delta V}{\delta \pi_X}(x_i) \\
&= f(x_i) - \frac{1}{\pi_X(x_i,t)} Q \nabla_x \pi_X(x_i,t) , \label{eq:ensembleFP}
\end{align}
$i=1,\ldots,M$. Contrary to the SDE (\ref{eq:MCSDE}), this formulation requires the PDF
$\pi_X(x,t)$, which is not explicitly available in general.  However, a
Gaussian approximation can be obtained from the available
ensemble $x_i(t)$, $i=1,\ldots,M$, using
\[
\pi_X (x,t) \approx \frac{1}{(2\pi)^{N/2}
  |P|^{1/2}}\exp\left(-\frac{1}{2} (x-\bar x_M(t))^{\rm T} P(t)^{-1} (x-\bar x_M(t))
\right)
\]
with empirical mean (\ref{eq:EME}) and empirical covariance matrix
\begin{equation} \label{eq:ECM}
P = \frac{1}{M-1} \sum_{i=1}^M (x_i-\bar x_M)(x_i-\bar x_M)^{\rm T} .
\end{equation}
Substituting this Gaussian approximation into (\ref{eq:ensembleFP})
yields the ensemble evolution equations 
\begin{equation} \label{eq:inflation}
\frac{{\rm d}x_i}{{\rm d}t} = f(x_i) + QP^{-1} (x_i-\bar x_M) ,
\end{equation}
which becomes exact in case the vector field $f$ is linear, \emph{i.e.}~$f(x) =
Ax + u$, the initial PDF $\pi_X(x,0)$ is Gaussian and for ensemble sizes $M\to \infty$.

We finally discuss the application of a particular type of SDEs
(\ref{eq:SDE}) as a way of generating samples $x_i$ from a given PDF
such as the posterior $\pi_X(x|y_0)$ of Bayesian inference. To do
this, consider the SDE (\ref{eq:SDE}) with the vector field $f$ being
generated by a potential $U(x)$, \emph{i.e.}~$f(x) = -\nabla_x U(x)$,
and $Q = I$. Then it can easily be verified that the PDF
\[
\pi^\ast_X(x) = Z^{-1} \exp (-U(x)), \qquad Z = \int_{\mathbb{R}^N}
\exp(-U(x)){\rm d}x,
\]
is stationary under the associated Fokker-Planck equation 
(\ref{eq:FokkerPlanck}). Indeed
\[
\nabla_x\cdot (\pi_X^\ast \nabla_X U) + \nabla_x \cdot
\nabla_x \pi_X^\ast = \nabla_x \cdot (\pi_X^\ast \nabla_x U + \nabla_x
\pi_X^\ast ) = 0 .
\]
Furthermore, it can be shown that $\pi_X^\ast$ is the unique
stationary PDF and that any initial PDF $\pi_X(t=0)$ approaches
$\pi_X^\ast$ at exponential rate under appropriate assumption on the 
potential $V$. Hence $X_t \sim \pi_X^\ast$ for $t \to
\infty$. This allows us to use an ensemble of solutions $x_i(t)$ of
(\ref{eq:MCSDE}) with an arbitrary initial PDF $\pi_X(x,0)$ 
as a method for generating ensembles from the prior 
or posterior Bayesian PDFs provided $U(x) = -\ln \pi_X(x)$ or $U(x) = -\ln \pi_X(x|y_0)$,
respectively. Note that the temporal dynamics of the associated SDE
(\ref{eq:SDE}) is not of any physical significance in this context
instead the SDE formulation is only taken as a device for generating
the desired samples. If the SDE formulation is replaced by the Euler-Maruyama method 
(\ref{eq:EMmethod}), time-stepping errors lead to sampling errors
which can be corrected for by combining (\ref{eq:EMmethod}) with a
Metropolis accept-reject criterion. The Metropolis adjusted method
gives rise to particular instances of \emph{Markov chain Monte Carlo
(MCMC) methods} such as the \emph{Metropolis adjusted Langevin algorithm (MALA)} or
the \emph{hybrid Monte Carlo (HMC) method}. The basic idea of MALA (as
well as HMC) is to rewrite (\ref{eq:EMmethod}) with $f(x) = -\nabla_x U(x)$, $Q=I$ as
\begin{align}
\label{eq:p step 1}
p_{n+1/2} &= p_n - \frac{1}{2} \sqrt{2\Delta t}\nabla_x U(x_n),\\
x_{n+1} &= x_n + \sqrt{2\Delta t} p_{n+1/2},\\
p_{n+1} & = p_{n+1/2} - \frac{1}{2} \sqrt{2\Delta t}\nabla_x U(p_n)
\label{eq:p step 2}
\end{align}
having introduced a dummy momentum variable $p$ with $p_n$ being a
realization of the random variable $Z_n \sim {\rm N}(0,I)$. Under the
Metropolis accept-reject criterion $x_{n+1}$ is accepted with
probability
\[
\min \{1,\exp (-(E_{n+1}-E_n))\},
\]
where
\[
E_n = \frac{1}{2} p_n^{\rm T} p_n + U(x_n), \qquad 
E_{n+1} = \frac{1}{2} p_{n+1}^{\rm T} p_{n+1} + U(x_{n+1})
\]
are the initial and final energies. Upon rejection one continues with
$x_n$.  The momentum value $p_{n+1}$ is discarded after a completed
time-step (regardless of its acceptance or rejection) and a new
momentum value is drawn from ${\rm N}(0,I)$. It should however be
noted that $|E_{n+1}-E_n| \to 0$ as the step-size $\Delta t$ goes to
zero and in practice the application of the Metropolis
accept-rejection step is often not necessary unless $\Delta t$ is
chosen too large. The HMC method differs from MALA in that several
iterations of (\ref{eq:p step 1}-\ref{eq:p step 2}) are applied before the
Metropolis accept-reject criterion is being applied.

%%%%%%%%%%%%%%%%%%%%%%%%%%%%%%%%%%%%%%%%%%%%%%%%%%%

\subsection*{References} 

A gentle introduction to stochastic processes can be found in \cite{sr:gardiner} and \cite{sr:Chorin}. 
A more mathematical treatment can be found in \cite{sr:StochProc,sr:SDEbook} and
numerical issues are discussed in \cite{sr:Higham,sr:Kloeden}. See \cite{sr:Otto01,sr:Villani}
for a discussion of the gradient flow structure of the Fokker-Planck equation. 
The ergodic behavior of Markov chains is covered in
\cite{sr:Meyn}. Markov chain Monte Carlo methods
and the hybrid Monte Carlo method in particular are treated in \cite{sr:Liu}. See also
\cite{sr:Roberts96} for the Metropolis adjusted Langevin algorithm (MALA).

%%%%%%%%%%%%%%%%%%%%%%%%%%%%%%%%%%%%%%%%%%%%%%%%%%%%%%%%%%%%%%%%%%
%%%%%%%%%%%%%%%%%%%%%%%%%%%%%%%%%%%%%%%%%%%%%%%%%%%%%%%%%%%%%%%%%

\section{Recent advances in data assimilation and filtering} \label{sec:4}

In this section, we combine Bayesian inference and stochastic
processes to tackle the problem of assimilating observational data
into scientific models.

\subsection{Preliminaries}

We select a model written as a time-discretized SDE, such as
(\ref{eq:EMmethod}), with the initial random variable satisfying $X_0
\sim \pi_0$. In addition to the pure prediction problem of computing
$\pi_n$, $n\ge 1$, for given $\pi_0$, we assume that model states $x
\in {\cal X} = \mathbb{R}^N$ are partially observed at equally spaced
instances in time.  These observations are to be assimilated into the
model.  More generally, \emph{intermittent data assimilation} is
concerned with fixed observation intervals $\Delta t_{\rm obs}>0$ and
model time-steps $\Delta t$ such that $\Delta t_{\rm obs} = L \Delta
t$, $L\ge 1$, which allows one to take the limit $L\to \infty$,
$\Delta t=\Delta t_{\rm obs}/L$. For simplicity, we will restrict the
discussion to the case where observations $y_0(t_n) = Y_n(\omega) \in
\mathbb{R}^K$ are made at every time step $t_n = n\Delta t$, $n\ge 1$
and the limit $\Delta t\to 0$ is not considered here.  We will further
assume that the observed random variables $Y_n$ satisfy the model
(\ref{eq:Ymodel}), i.e.
\[
Y_n = h(X_n) + \Xi_n
\]
and the measurement errors $\Xi_n \sim {\rm N}(0,R)$ are mutually
independent with common error covariance matrix $R$.  
We introduce the notation ${\rm Y}_k = \{y_0(t_i)\}_{i=1,\ldots,k}$ to denote 
all observations up to and including time $t_k$.
\begin{Definition*}[Data assimilation] \emph{Data assimilation} is the
  estimation of marginal PDFs $\pi_n(x|{\rm Y}_k)$ of the random
  variable $X_n = X(t_n)$ conditioned on the set of observations ${\rm
    Y}_k$. We distinguish three cases: (i) \emph{filtering} $k=n$, (ii)
  \emph{smoothing} $k>n$, and (iii) \emph{prediction} $k<n$.
\end{Definition*}

The subsequent discussions are restricted to the filtering problem. We
have already seen that evolution of the marginal distributions under
(\ref{eq:EMmethod}) alone is governed by the Chapman-Kolmogorov
equation (\ref{eq:CK}) with transition probability density
(\ref{eq:CKFP}). We denote the associated Frobenius-Perron operator
(\ref{eq:Frobenius}) by ${\cal P}_{\Delta t}$. Given $X_0 \sim \pi_0$,
we first obtain
\[
\pi_1 = {\cal P}_{\Delta t}\pi_0.
\]
This time propagated PDF is used as the prior PDF $\pi_X = \pi_1$ in Bayes' formula (\ref{eq:Bayes})
at $t=t_1$ with $y_0 = y_0(t_1)$ and likelihood
\[
\pi_Y(y|x) = \frac{1} {(2\pi)^{N/2} |R|^{1/2}} \exp \left(-\frac{1}{2}
  (y-h(x))^{\rm T} R^{-1} (y-h(x)) \right).
\]
Bayes' formula implies the posterior PDF
\[
\pi_1(x|{\rm Y}_1) \propto \pi_Y(y_0(t_1)|x) \pi_1(x) ,
\]
where the constant of proportionality depends on $y_0(t_1)$
only. 

\begin{Proposition*}[Sequential filtering]
The filtering problem leads to the recursion
\begin{equation} \label{eq:RRR}
\begin{array}{rcl} 
\pi_{n+1}(\cdot |{\rm Y}_n) &=&
{\cal P}_{\Delta t} \pi_n(\cdot|{\rm Y}_n),\\
\pi_{n+1}(x|{\rm Y}_{n+1}) &\propto& \pi_Y(y_0(t_{n+1})|x)
\,\pi_{n+1}(x|{\rm Y}_n),   \end{array} 
\end{equation}
$n\ge 0$, and $X_n \sim \pi_n(\cdot|{\rm Y}_n)$ solves the filtering
problem at time $t_n$. The constant of proportionality depends on
$y_0(t_{n+1})$ only.
\end{Proposition*}

\begin{proof} The recursion follows by induction.
\end{proof}

Recall that the Frobenius-Perron operator ${\cal P}_{\Delta t}$ is
generated by the stochastic diffference equation
(\ref{eq:EMmethod}). On the other hand, Bayes' formula only leads to a
transition from the predicted $\pi_{n+1}(x|{\rm Y}_n)$ to the filtered
$\pi_{n+1}(x|{\rm Y}_{n+1})$. Following our discussion on transport
maps from Section \ref{sec:OT}, we assume the existence of a transport
map $X' = T_{n+1} (X)$, depending on $y_0(t_{n+1})$, that
couples the two PDFs. The use of optimal transport maps in the context
of Bayesian inference and intermittent data assimilation was first
proposed in \cite{sr:reich10,sr:marzouk11}.

\begin{Proposition*}[Filtering by transport maps] Assuming the
  existence of appropriate transport maps $T_{n+1}$, which couple
  $\pi_{n+1}(x|{\rm Y}_n)$ and $\pi_{n+1}(x|{\rm Y}_{n+1})$, the
  filtering problem is solved by the following recursion for the
  random variables $X_{n+1}$, $n\ge 0$:
\begin{equation} \label{eq:RR}
X_{n+1} = T_{n+1} \left(X_n + \Delta t f(X_n) + \sqrt{2\Delta t}
  Z_n\right) ,
\end{equation}
which gives rise to a time-dependent Markov process.
\end{Proposition*}

\begin{proof} Follows trivially from (\ref{eq:RRR}).
\end{proof}

The rest of this section is devoted to several Monte Carlo
methods for sequential filtering. 

%%%%%%%%%%%%%%%%%%%%%%%%%%%%%%%%%%%%%%%

\subsection{Sequential Monte Carlo method} \label{sec:SMCM}

In our framework, a standard sequential Monte Carlo method, also
called \emph{bootstrap particle filter}, may be described as an
ensemble of random variables $X_i$ and associated realizations
(referred to as ``particles'') $x_i = X_i(\omega)$, which follow the
stochastic difference equation (\ref{eq:EMmethod}), choosing the
transport map in (\ref{eq:RR}) to be the identity map. Observational
data is taken into account using importance sampling as discussed in
Section \ref{sec:MC}, i.e., each particle carries a weight $w_i(t_n)$,
which is updated according to Bayes' formula
\[
w_i(t_{n+1}) \propto w_i(t_n) \pi(y_0(t_{n+1})|x_i(t_{n+1})).
\]
The constant of proportionality is chosen such that the new weights $\{w_i(t_{n+1})\}_{i=1}^M$ 
sum to one.

Whenever the particle weights $w_i(t_n)$ start to become highly
non-uniform (or possibly also after each assimilation step) resampling
 is necessary in order to generate a new family of random variables with equal
weights. 

Most available resampling methods start from the
\emph{weighted empirical measure}
\begin{equation} \label{eq:empirical}
\mu_X({\rm d}x) = \sum_{i=1}^M w_i \mu_{x_i}({\rm d}x)
\end{equation}
associated with a set of weighted samples $\{x_i,w_i\}_{i=1}^M$. The
idea is to replace each of the original samples $x_i$ by $\xi_i\ge 0$
offsprings with equal weights $\hat w_i = 1/M$. The distribution of
offsprings is chosen to be equal to the distribution of $M$ samples
(with replacement) drawn at random from the empirical distribution
(\ref{eq:empirical}). In other words, the offsprings
$\{\xi_i\}_{i=1}^M$ follow a \emph{multinomial distribution} defined
by
\begin{equation} \label{eq:multinomial}
\mathbb{P}(\xi_i= n_i, i=1,\ldots,M) = \frac{M!}{\prod_{i=1}^M n_i !}
\prod_{i=1}^M (w_i)^{n_i} 
\end{equation}
with $n_i\ge 0$ such that $\sum_{i=1}^M n_i = M$. 
In practice, independent resampling is often replaced by residual or
systematic resampling. We next summarize residual resampling while we
refer the reader to \cite{sr:arul02} for an algorithmic description of
systematic resampling.

\begin{Definition*}[Residual resampling] \emph{Residual resampling}
  generates
\[
\xi_i = \lfloor M w_i \rfloor + \bar \xi_i,
\]
offsprings of each ensemble member $x_i$ with weight $w_i$,
$i=1,\ldots,M$.  Here $\lfloor x \rfloor$ denotes the integer part of
$x$ and $\bar \xi_i$ follows the multinomial distribution
(\ref{eq:multinomial}) with weights $w_i$ being replaced by
\[
\overline w_i = \frac{Mw_i - \lfloor M w_i \rfloor }{ \sum_{j=1}^M (Mw_j -
  \lfloor M w_j \rfloor )} 
\]
and with a total of 
\[
\sum_{i=1}^M n_i = \overline M := M - \sum_i \lfloor M w_i\rfloor 
\]
independent trials.
\end{Definition*}

In generalization of (\ref{eq:multinomial}),
we introduce the notation ${\rm Mult}(L;\omega_1,\ldots,\omega_M)$ to
denote the multinomial distribution of $L$ independent trials, where
the outcome of each trial is distributed among $M$ possible outcomes according to
probabilities $\{\omega_i\}_{i=1}^M$. The following algorithm draws random
samples from ${\rm Mult}(L;\omega_1,\ldots,\omega_M)$. We first
introduce the generalized inverse cumulative distribution function $F_{\rm
  emp}^{-1} :[0,1] \to \{1,\ldots,M\}$ for the empirical measure
(\ref{eq:empirical}), which is defined by
\[
F_{\rm emp}^{-1} (u) = i \quad \Longleftrightarrow \quad u \in \left(
  \sum_{j=1}^{i-1} \omega_j,\sum_{j=1}^i \omega_i \right] .
\]
We next draw $L$ independent samples $u_l\in [0,1]$ from the uniform
distribution ${\rm U}[0,1]$ and initially set the number of copies
$\bar \xi_i$, $i=1,\ldots,M$, equal to zero. For $l=1,\ldots,L$, we now
increment $\bar \xi_{I_l}$ by one for indices $I_l \in \{1,\ldots,M \}$, $l=1,\ldots,L$,
defined by 
\[
I_l = F_{\rm emp}^{-1} (u_l) 
= \arg \min_{i\ge 1} \sum_{j=1}^i
\omega_j \ge u_l .
\]

Both independent and residual resampling can be viewed as providing a coupling
between the empirical measure (\ref{eq:empirical}) will all weights
being equal to $w_i = 1/M$ and the target measure (\ref{eq:empirical})
with identical samples $\{x_i\}$ but non-uniform weights. Clearly
residual resampling provides a coupling with a smaller transport cost. 
This can already be concluded from the trivial case of equal weights
in the target measure in which case residual resampling reduces to the
identity map with zero transport cost while independent resampling remains
non-deterministic and produces a non-zero transport cost. The following example outlines
the optimal transportation perspective on resampling more precisely for two discrete, univariate 
random variables.

\begin{Example*}[Coupling discrete random variables]
Let us consider two discrete, univariate random variables $X_i:\Omega \to {\cal X}$,
$i=1,2$, with target set
\[
{\cal X} = \{x_1,x_2,\ldots,x_M\} \in \mathbb{R}^M.
\]
We furthermore assume that 
\[
\mathbb{P}(X_1(\omega) = x_i) = 1/M, \qquad 
\mathbb{P}(X_2(\omega) = x_i) = w_i 
\]
for given probabilities/weights $w_i\ge 0$, $i=1,\ldots,M$. Any coupling of $X_1$ and $X_2$ is characterized
by a matrix ${\cal T} \in \mathbb{R}^{M\times M}$ such that $t_{ij} =
({\cal T})_{ij}\ge 0$ and
\[
\sum_{i=1}^M t_{ij} = 1/M, \qquad \sum_{j=1}^M t_{ij} = w_i
\]
Given a coupling ${\cal T}$ and the mean values 
\[
\bar x_1 = \frac{1}{M}\sum_i x_i, \qquad \bar x_2 = \sum_i w_i x_i
\]
the covariance between $X_1$ and $X_2$ is defined by
\[
\mbox{cov}(X_1,X_2) = \sum_{i,j} (x_i-\bar x_2) t_{ij} (x_j-\bar x_1).
\]
The induced Markov transistion matrix from $X_1$ to $X_2$ is simply given by $M{\cal T}$.
Independent resampling corresponds to $t_{ij} = w_i/M$ and leads to a zero
correlation between $X_1$ and $X_2$. On the other hand, maximizing the correlation results
in a linear programming problem for the $M^2$ unknowns $\{t_{ij}\}$. Its solution then also defines 
the solution to the associated optimal transportation problem (\ref{eq:MK}).
\end{Example*}

More generally, sequential Monte Carlo methods differ by
the way resampling is implemented and also in the choice of proposal
step which in our context amounts to choosing transport maps $T_{n+1}$
in (\ref{eq:RR}) which are different from the identity map. See also
the discussion in Section \ref{sec:IS} below.

%%%%%%%%%%%%%%%%%%%%%%%%%%%%%%%%%%%%%%%%%%%%%

\subsection{Ensemble Kalman filter (EnKF)}
\label{s:enkf}

We now introduce an alternative to sequential Monte Carlo methods
which has become hugely popular in the geophysical community in recent
years. The idea is to construct a simple but robust transport map
$T'_{n+1}$ which replaces $T_{n+1}$ in (\ref{eq:RR}). This transport
map is based on the \emph{Kalman update
equations} for linear SDEs and Gaussian prior and posterior
distributions.  We recall the standard Kalman filter update equations.

\begin{Proposition*}[Kalman update for Gaussian distributions] Let the
  prior distribution $\pi_X$ be Gaussian with mean $\bar x^f$ and
  covariance matrix $P^f$. Observations $y_0$ are assumed to follow
  the linear model 
\[
Y = HX + \Xi,
\]
where $\Xi \sim {\rm N}(0,R)$ and $R$ is a symmetric,
positive-definite matrix. Then the posterior distribution
$\pi_X(x|y_0)$ is also Gaussian with mean
\begin{equation} \label{eq:Kalman1}
\bar x^a = \bar x^f - P^fH^{\rm T} (HP^f H^{\rm T} + R)^{-1} (H\bar x^f - y_0)
\end{equation}
and covariance matrix
\begin{equation} \label{eq:Kalman2}
P^a = P^f - P^f H^{\rm T} (HP^f H^{\rm T} + R)^{-1} H P^f.
\end{equation}
Here we adopt the standard meteorological notation with superscript $f$
(forecast) denoting prior statistics, and superscript $a$ (analysis)
denoting posterior statistics.
\end{Proposition*}

\begin{proof} By straightforward generalization to vector-valued
  observations of the case of a scalar observation already discussed
  in Section \ref{sec:Bayes}.
\end{proof}

EnKFs rely on the assumption that the predicted PDF $\pi_{n+1}(x|{\rm
  Y}_n)$ is approximately Gaussian. The ensemble $\{x_i\}_{i=1}^M$ of
model states is used to estimate the mean and the covariance matrix
using the empirical estimates (\ref{eq:EME}) and (\ref{eq:ECM}),
respectively. The key novel idea of EnKFs is to then interpret the
posterior mean and covariance matrix in terms of appropriately
adjusted ensemble positions. This adjustment can be thought of as a
coupling of the underlying prior and posterior random variables of
which the ensembles are realizations.  The original EnKF
\cite{sr:burgers98} uses perturbed observations to achieve the desired
coupling.

\begin{Definition*}[Ensemble Kalman Filter]
  The \emph{EnKF with perturbed observations} for a linear observation
  operator $h(x) = Hx$ is given by
\begin{align} \label{eq:EnKF1a}
X_{n+1}^f &= X_n + \Delta t f(X_n) + \sqrt{2\Delta
  t} Z_n,\\ \label{eq:EnKF1b}
X_{n+1} &= X_{n+1}^f - P_{n+1}^fH^{\rm T}(HP_{n+1}^fH^{\rm T} + R)^{-1}
(HX_{n+1}^f - y_0 + \Sigma_{n+1}),
\end{align}
where the random variables $Z_n \sim {\rm N}(0,Q)$, $\Sigma_{n+1} \sim
{\rm N}(0,R)$ are the mutually independent perturbations to the
observations, $y_0 = y_0(t_{n+1})$, $\bar x_{n+1}^f
=\mathbb{E}_{X_{n+1}^f}[x]$, and
\[
P_{n+1}^f =
\mathbb{E}_{X_{n+1}^f}[(x-\bar x_{n+1}^f)(x-\bar x_{n+1}^f)^{\rm T}].
\]
\end{Definition*}

Next, we investigate the properties of the assimilation step
(\ref{eq:EnKF1b}).

\begin{Proposition*}[EnKF consistency] The EnKF update step
  (\ref{eq:EnKF1b}) propagates the mean and covariance matrix of $X$
  in accordance with the Kalman filter equations for Gaussian PDFs.
\end{Proposition*}

\begin{proof}
  It is easy to verify that the ensemble mean satisfies
\[
\bar x_{n+1} = \bar x_{n+1}^f -  P_{n+1}^fH^{\rm T}(HP_{n+1}^fH^{\rm T} + R)^{-1}
(H\bar x_{n+1}^f - y_0 ),
\]
which is consistent with the Kalman filter update for the ensemble
mean. Furthermore, the deviation $\delta X = X-\bar x$ satisfies
\[
\delta X_{n+1} = \delta X_{n+1}^f - P_{n+1}^fH^{\rm
  T}(HP_{n+1}^fH^{\rm T} +
R)^{-1} (H\delta X_{n+1}^f + \Sigma_{n+1}),
\]
which implies
\begin{align*}
P_{n+1} &= P_{n+1}^f - 2 P_{n+1}^f H^{\rm T}
(HP_{n+1}^f H^{\rm T} + R)^{-1} H P_{n+1}^f + \\
& \qquad P_{n+1}^f H^{\rm T} 
(HP_{n+1}^f H^{\rm T} + R)^{-1} R (HP_{n+1}^f H^{\rm T} + R)^{-1} H P_{n+1}^f  + \\
&\qquad (HP_{n+1}^f H^{\rm T} + R)^{-1} H P_{n+1}^f H^{\rm T}
(HP_{n+1}^f H^{\rm T} +
R)^{-1} H P_{n+1}^f \\
&= P_{n+1}^f -  P_{n+1}^f  H^{\rm T} (HP_{n+1}^f H^{\rm T} + R)^{-1} H P_{n+1}^f 
\end{align*}
for the update of the covariance matrix, which is also consistent with
the Kalman update step for Gaussian random variables.
\end{proof}
 
Practical implementations of the EnKF with perturbed observations
replace the exact mean and covariance matrix by ensemble based
empirical estimates (\ref{eq:EME}) and (\ref{eq:ECM}), respectively.

Alternatively, we can derive a transport map $T$ under the assumption
of Gaussian prior and posterior distributions as follows. Using the
empirical ensemble mean $\bar x$ we define ensemble deviations by
$\delta x_i = x_i - \bar x \in \mathbb{R}^N $ and an associated
ensemble deviation matrix $\delta {\rm X} = (\delta x_1,\ldots,\delta
x_M) \in \mathbb{R}^{N\times M}$.  Using the notation, the empirical
covariance matrix of the prior ensemble at $t_{n+1}$ is then given by
\[
P_{n+1}^f = \frac{1}{M-1} \delta {\rm X}_{n+1}^f \,(\delta {\rm
  X}_{n+1}^f)^{\rm T}
\]
We next seek a matrix $S\in \mathbb{R}^{M\times M}$ such that
\[
P_{n+1} = \frac{1}{M-1} \delta {\rm X}_{n+1}^f S S^{\rm T} (\delta {\rm
  X}_{n+1}^f)^{\rm T},
\]
where the rows of $S$ sum to zero in order to preserve the zero mean
property of $\delta {\rm X}_{n+1} = \delta {\rm X}_{n+1}^f S$.  Such
matrices do exist (see \emph{e.g.} \cite{sr:evensen}) and give rise the \emph{ensemble square
root filters}. More specifically, Kalman's update formula
(\ref{eq:Kalman2}) for the posterior covariance matrix implies 
\begin{align*}
P^a 
& = \frac{1}{M-1} \delta {\rm X}^f \left\{ I -  \frac{1}{M-1} (\delta {\rm
    Y}^f)^{\rm T}  \left[ H P^f H^{\rm T} +R \right]^{-1}  \delta {\rm
  Y}^f \right\} (\delta {\rm X}^f)^{\rm T} \\
&= \frac{1}{M-1} \delta {\rm X}^f S S^{\rm T} (\delta {\rm
  X}^f)^{\rm T} ,
\end{align*}
where we have dropped the time index subscript and introduced the
ensemble perturbations $\delta {\rm Y}^f = H \delta {\rm X}^f$ in
observation space ${\cal Y}$. Recalling now the definition of a matrix
square root from Section \ref{sec:OT} and making use of the
Sherman-Morrison-Woodbury formula \cite{sr:golub}, we find that
\begin{align} \nonumber
S &=  \left\{ I - \frac{1}{M-1}
  (\delta {\rm Y}^f)^{\rm T}   \left[ H P^f H^{\rm T} + R \right]^{-1}   \delta {\rm
  Y}^f \right\}^{1/2} \\
&= \left\{ I + \frac{1}{M-1}  (\delta {\rm Y}^f)^{\rm T} R^{-1}  \delta {\rm
  Y}^f  \right\}^{-1/2}. \label{eq:Stransform}
\end{align}
The complete ensemble update of an ensemble square root
filter is then given by
\begin{equation} \label{eq:EnKF2a}
x_i(t_{n+1}) = \bar x_{n+1} + \delta X_{n+1}^f S e_i,
\end{equation}
where $e_i$ denotes the $i$th basis vector in
$\mathbb{R}^M$ and
\begin{equation*} %\label{eq:EnKF2b}
\bar x_{n+1} = \bar x^f_{n+1} - P_{n+1}^f H^{\rm T} (HP_{n+1}^fH^{\rm T} + R)^{-1} (H\bar x^f_{n+1} -
y_0(t_{n+1}))
\end{equation*}
denotes the updated ensemble mean. 

We now discuss the update (\ref{eq:EnKF2a}) from the perspective
of optimal transportation which in our context reduces to finding a
matrix $S_{\rm OT} \in \mathbb{R}^{M\times M}$ such that the trace of 
\begin{equation*}
\mbox{cov}(\delta X_{n+1}^f,\delta X_{n+1}) = \mathbb{E}[\delta
X_{n+1}^f S_{\rm OT}^{\rm T} (\delta X_{n+1}^f)^{\rm T})
\end{equation*}
is maximized. 

\begin{Proposition*}[Optimal update for ensemble square root filter]
The trace of the covariance matrix ${\rm cov}(\delta X_{n+1}^f,\delta
X_{n+1})$ is maximized for
\[
\delta X_{n+1} = \delta X_{n+1}^f S_{\rm OT}
\]
with transform matrix
\[
S_{\rm OT} = \frac{1}{\sqrt{M-1}} S \left[ S (\delta X_{n+1}^f)^{\rm T} P^f \delta X_{n+1}^f S
  \right]^{-1/2} S (\delta X_{n+1}^f)^{\rm T} \delta X_{n+1}^f
\]
and $S \in \mathbb{R}^{M\times M}$ given by (\ref{eq:Stransform}).
\end{Proposition*}

\begin{proof}
Follows from (\ref{eq:OGcoupling2}) with $A = \delta X_{n+1}^f
S/\sqrt{M-1}$ and $\Sigma_1 = P^f$. The left multiplication in
(\ref{eq:OGcoupling}) is finally rewritten as a right multiplication by
$S_{\rm OT} \in \mathbb{R}^{M\times M}$ in terms of ensemble deviations $\delta X_{n+1}^f$. 
\end{proof}

We finish this section by brief discussions on a couple of practical issues.
It is important to recall that the Kalman filter can be viewed as a
linear minimum variance estimator \cite{sr:duncan72}. At the same time 
it has been noted \cite{sr:xiong06,sr:lei11} that the updated ensemble
mean $\bar x_{n+1}$ is biased in case the prior distribution is not Gaussian.
Hence the associated mean squared error (\ref{eq:MSE}) does not vanish
as $M\to \infty$ even though the variance of the estimator goes to
zero. If desired the bias can be removed by replacing $\bar x_{n+1}$ in (\ref{eq:EnKF2a}) by
(\ref{meanposterior}) with weights (\ref{weightsposterior}), where
$y_0 = y_0(t_{n+1})$ and $x_i^{\rm  prior} =
x_i^f(t_{n+1})$. Higher-order moment corrections can also be
implemented \cite{sr:xiong06,sr:lei11}. However, the filter
performance only improves for sufficiently large ensemble sizes.

We mention the \emph{unscented Kalman filter} \cite{sr:Julier97anew} 
as an alternative extension of the Kalman filter to nonlinear
dynamical systems. We also mention the \emph{rank histogram filter} \cite{sr:anderson10},
which is based on first constructing an approximative coupling in the observed
variable $y$ alone followed by linear regression of the updates in
$y$ onto the state space variable $x$. 

Practical implementations of EnKFs for high-dimensional problem rely
on additional modifications, in particular \emph{inflation} and
\emph{localization}. While localization modifies the covariance matrix
$P^f$ in the Kalman update (\ref{eq:EnKF1b}) in order to increase its
rank and to localize the spatial impact of observations in physical
space, inflation increases the ensemble spread $\delta x_i = x_i -
\bar x$ by replacing $x_i$ by $\bar x + \alpha (x_i-\bar x)$ with
$\alpha > 1$.  Note that the second term on the righthand side of
(\ref{eq:inflation}) achieves a similar effect and ensemble inflation
can be viewed as simple parametrization of (stochastic) model
errors. See \cite{sr:evensen} for more details on inflation and
localization techniques.

%%%%%%%%%%%%%%%%%%%%%%%%%%%%%%%%%%%%%%%%%%%%%%%%%%%%%%%

\subsection{Ensemble transform Kalman-Bucy filter} \label{sec:KB}

In this section, we describe an alternative implementation of ensemble square root filters based
on the Kalman-Bucy filter. We first describe the Kalman-Bucy formulation of the linear filtering
problem for Gaussian PDFs.

\begin{Proposition*}[Kalman-Bucy equations]
  The Kalman update step (\ref{eq:Kalman1})-(\ref{eq:Kalman2}) can be
  formulated as a differential equation in artificial time $s\in
  [0,1]$. The Kalman-Bucy equations are
\[
\frac{{\rm d}\bar x}{{\rm d}s} = -PH^{\rm T} R^{-1} (H\bar x - y_0)
\]
and
\[
\frac{{\rm d}P}{{\rm d}s} = -  PH^{\rm T} R^{-1} HP.
\]
The initial conditions are $\bar x(0) = \bar x^f$ and $P(0) = P^f$ and
the Kalman update is obtained from the final conditions $\bar x^a =
\bar x(1)$ and $P^a = P(1)$.
\end{Proposition*}

\begin{proof} We present the proof for $N=1$ (one dimensional state
  space) and $K=1$ (a single observation). Under this assumption, the
  standard Kalman analysis step (\ref{eq:Kalman1})-(\ref{eq:Kalman2})
  gives rise to
\begin{equation*}
P^a = \frac{P^f R }{P^f + R},
\qquad 
\bar{x}^a = \frac{\bar{x}^f R + y_0  P^f}{P^f + R},
\end{equation*} 
for a given observation value $y_0$. 

We now demonstrate that this update is equivalent to twice the application
of a Kalman analysis step with $R$ replaced by $2R$. Specifically, we obtain
\begin{equation*}
\hat P^a = \frac{2P_m R }{P_m + 2R}, \qquad 
P_m = \frac{2P^f R }{P^f + 2R},
\end{equation*}
for the resulting covariance matrix $\hat P^a$ with intermediate value
$P_m$. The analyzed mean $\hat{x}^a$ is provided by
\begin{equation*}
\hat{x}^a = \frac{2 \bar{x}_m R + y_0 P_m}{P_m + 2R}, \qquad 
\bar{x}_m = \frac{2 \bar{x}^f R + y_0 P^f}{P^f + 2R} .
\end{equation*}
We need to demonstrate that $P^a = \hat P^a$ and $\bar{x}^a = 
\hat{x}^a$. We start with the covariance matrix and obtain
\begin{align*}
\hat P^a = \frac{\frac{4P^f R}{P^f + 2R} R}{\frac{2P^f R}{P^f + 2R} + 2R}
= \frac{4P^f R^2}{4P^f R + 4R^2} = \frac{P^f R}{P^f + R} = P^a.
\end{align*}
A similar calculation for $\hat{x}^a$ yields
\begin{align*}
\hat{x}^a = \frac{2 \frac{2\bar{x}^f R + y_0 P^f}{P^f + 2R} R
+ y_0 \frac{2P^f R}{P^f + 2R}}{2R + \frac{2P^f R}{P^f + 2R}}
= \frac{4\bar{x}^f R^2 + 4y_0 P^f R}{4R^2 + 4R P^f} = \bar{x}^a .
\end{align*}
Hence, by induction, we can replace the standard Kalman analysis step
by $D>2$ iterative applications of a Kalman analysis with $R$ replaced
by $D R$. We set $P_0 = P^f$, $\bar{x}_0 = \bar{x}^f$, and
iteratively compute $P_{j+1}$ from
\begin{equation*}
P_{j+1} = \frac{D P_j R}{ P_j + D R}, \qquad 
\bar{x}_{j+1} = \frac{D \bar{x}_j R + y_0 P_j}{P_j + D R}
\end{equation*}
for $j=0,\ldots,D-1$. We finally set $P^a = P_D$ and $\bar{x}^a = 
\bar{x}_D$. Next we introduce a step-size $\Delta s = 1/D$ and assume
$D\gg 1$. Then
\begin{align*}
\overline{x}_{j+1} = \frac{\bar{x}_j R + \Delta s y_0 P_j}{R + 
\Delta s P_j} 
= \bar{x}_j - \Delta s P_j R^{-1} \left( \bar{x}_j 
- y_0 \right) + {\cal O}(\Delta s^2)
\end{align*}
as well as
\begin{equation*}
P_{j+1} = \frac{P_j R}{R + \Delta s P_j} =
P_j - \Delta s P_j R^{-1} P_j + {\cal O}(\Delta s^2).
\end{equation*}
Taking the limit $\Delta s \to 0$, we obtain the two differential
equations
\begin{equation*}
\frac{{\rm d}P}{{\rm d} s}  = - P R^{-1} P , \qquad
\frac{{\rm d} \overline{x} }{{\rm d} s}  = -P R^{-1} \left( 
\overline{x} - y_0 \right)
\end{equation*}
for the covariance and mean, respectively. The equation for $P$ can be 
rewritten in terms of its square root $Y$ (\emph{i.e.}~$P = Y^2)$ as 
\begin{equation} \label{eq:squareroot}
\frac{{\rm d} Y}{{\rm d} s}  = -\frac{1}{2} P R^{-1} Y.
\end{equation}
\end{proof}

Upon formally setting $Y = \delta {\rm X}/\sqrt{M-1}$ in
(\ref{eq:squareroot}), the Kalman-Bucy filter
equations give rise to a particular implementation of ensemble square root 
filters in terms of evolution equations in artificial time $s\in [0,1]$.

\begin{Definition*}[Ensemble transform Kalman-Bucy filter equations]
  The \emph{ensemble transform Kalman-Bucy filter} equations
  \cite{sr:br10b,sr:br11,sr:akir11} for the assimilation of an observation $y_0 = y_0(t_n)$ at
 $t_n$ are given by
\[
\frac{{\rm d}x_i}{{\rm d}s} = -\frac{1}{2} P H^{\rm T} R^{-1} (Hx_i + H\bar x - 2y_0(t_n))
\]
in terms of the ensemble members $x_i$, $i=1,\ldots,M$ and are solved
over a unit time interval in artificial time $s \in [0,1]$. Here $P$
denotes the empirical covariance matrix (\ref{eq:ECM}) and $\bar x$
the empirical mean (\ref{eq:EME}) of the ensemble. 
\end{Definition*}

The Kalman-Bucy
equations are realizations of an underlying differential equation
\begin{equation} \label{eq:BR1}
\frac{{\rm d}X}{{\rm d}s} = -\frac{1}{2}P H^{\rm T} R^{-1} (HX + H\bar x - 2y_0(t_n))
\end{equation}
in the random variable $X$ with mean
\[
\bar x = \mathbb{E}_X[x] = \int_\mathbb{R}^N x\pi_X {\rm d}x
\]
and covariance matrix
\[
P = \mathbb{E}_X[(x-\bar x)(x-\bar x)^{\rm T} ].
\]

The associated evolution of the PDF $\pi_X$ (here assumed to be
absolutely continuous) is given by Liouville's equation
\begin{equation} \label{eq:Liouvilles}
\frac{\partial \pi_X}{\partial s} = -\nabla_x \cdot \left( \pi_X v \right)
\end{equation}
with vector field
\begin{equation} \label{eq:Vel}
v(x) = -\frac{1}{2}P H^{\rm T} R^{-1} (Hx + H\bar x - 2y_0(t_n)).
\end{equation}
Recalling the earlier discussion of the Fokker-Planck equation in
Section \eqref{s:sdes}, we note that (\ref{eq:Liouvilles}) with vector field (\ref{eq:Vel}) 
also has an interesting geometric structure.

\begin{Proposition*}[Ensemble transform Kalman-Bucy equations as a
  gradient flow] The vector field (\ref{eq:Vel}) is equivalent to
\[
v(x) = - P \nabla_x  \frac{\delta F}{\delta \pi_X}
\]
with potential
\begin{align} \nonumber
F(\pi_X) &= \frac{1}{4} \int_{\mathbb{R}^N} (Hx-y_0(t_n))^{\rm T} R^{-1}(
Hx - y_0(t_n)) \pi_X {\rm d}x +\\
& \qquad \quad
\frac{1}{4} (H\bar x-y_0(t_n))^{\rm T} R^{-1}(
H\bar x - y_0(t_n))  . \label{eq:KBFunctional}
\end{align}
Liouville's equation (\ref{eq:Liouvilles}) can be stated as
\[
\frac{\partial \pi_X}{\partial s} = -\nabla_x \cdot \left( \pi_X v 
\right) = -\mbox{\rm grad}_{\pi_X} F(\pi_X),
\]
where we have used ${\rm M} = P$ in the definition of the gradient (\ref{gradient}).
\end{Proposition*}

\begin{proof} The result can be verified by direct calculation.
\end{proof}

Nonlinear forward operators can be treated in this framework by replacing the potential
(\ref{eq:KBFunctional}) by, for example,
\begin{align*} 
F(\pi_X) &= \frac{1}{4} \int_{\mathbb{R}^N} (h(x)-y_0(t_n))^{\rm T} R^{-1}(
h(x) - y_0(t_n)) \pi_X {\rm d}x +\\
& \qquad \quad
\frac{1}{4} (h(\bar x)-y_0(t_n))^{\rm T} R^{-1}(
H\bar x - y_0(t_n))  . 
\end{align*}
Efficient time-stepping methods for the ensemble transform 
Kalman-Bucy filter equations are discussed in \cite{sr:akir11} and an application
to continuous data assimilation can be found in \cite{sr:br11}.

%%%%%%%%%%%%%%%%%%%%%%%%%%%%%%%%%%%%%%%%%%%%%%%%

\subsection{Guided sequential Monte Carlo methods} \label{sec:IS}

EnKF techniques are limited by the fact that the empirical PDFs do not
converge to the filter solution in the limit of ensemble sizes $M\to
\infty$ unless the involved PDFs are Gaussian. Sequential Monte Carlo
methods, on the other hand, lead to unbiased estimators for the mean
and can be shown to converge under fairly general assumptions, but
they do not work well in high-dimensional phase spaces since
importance sampling is not sufficient to guarantee good performance of a particle
filter for finite ensemble sizes. In particular, the variance in the associated
mean squared error (\ref{eq:MSE}) can be very large for ensemble sizes
typically used in geophysical applications.

The combination of modified particle positions and
appropriately ajdusted particle weights appears therefore as a
promising area for research and might achieve a better bias-variance
tradeoff than either the EnKF or traditional sequential Monte Carlo
methods. In particular, combining ensemble transform
techniques, such as EnKF, with sequential Monte Carlo methods appears as a natural
research direction. Indeed, in the framework of Monte Carlo methods discussed
in Section \ref{sec:MC}, the standard sequential Monte Carlo approach
consists of importance sampling using proposal PDF $\pi_X'(x) = \pi_{n+1}(x|{\rm
  Y}_n)$ and subsequent reweighting of particles according to (\ref{weightsposterior}). As
also already discussed in Section \ref{sec:MC}, the performance of
importance sampling can be improved by applying modified proposal
densities $\pi_{n+1}'(x|{\rm Y}_{n+1})$ with the aim of pushing the
updated ensemble members $x_i(t_{n+1})$ to regions of high and nearly equal probability in the
targeted posterior PDF $\pi_{n+1}(x|{\rm Y}_{n+1})$ (compare with
eq.~(\ref{eq:importantweights})). We call the resulting filter
algorithms \emph{guided sequential Monte Carlo methods}. 

More precisely, a guided sequential Monte Carlo method is defined by a 
conditional proposal PDF $\tilde \pi_{n+1}(x'|x,y_0(t_{n+1}))$ and
an associated joint PDF
\begin{equation} \label{eq:jointproposal}
\tilde \pi_{X'X}(x',x|{\rm Y}_{n+1}) = \tilde \pi_{n+1}(x'|x,y_0(t_{n+1}))\,\pi_n(x|{\rm Y}_n) .
\end{equation}
An ideal proposal density (in the sense of coupling) would lead to a marginal distribution
$\tilde \pi_{X'}(x|{\rm Y}_{n+1})$, which is identical to the posterior PDF $\pi_{n+1}(x|{\rm
  Y}_{n+1})$. In guided sequential Monte Carlo methods, a mismatch
between $\pi_{X'}(x|{\rm Y}_{n+1})$ and $\tilde \pi_{n+1}(x|{\rm Y}_{n+1})$ is treated by
adjusted particle weights $w_i(t_{n+1})$. Following the general methodology of
importance sampling one obtains the recursion
\[
w_i(t_{n+1}) \propto \frac{\pi_Y(y_0(t_{n+1})|x'_i)
  \pi(x_i'|x_i)}{\tilde \pi_{n+1}(x_i'|x_i,y_0(t_{n+1}))} w_i(t_n).
\]
Here $\pi(x'|x)$ denotes the conditional PDF (\ref{eq:CKFP})
describing the model dynamics,
$(x_i',x_i)$, $i=1,\ldots,M$, are realizations from the joint PDF
(\ref{eq:jointproposal}) with weights $w_i(t_n)$, $x_i = x_i(t_n)$, and the approximation 
\begin{align*} 
\mathbb{E}_{X_{n+1}}[g] & = \frac{1}{\pi_Y(y_0(t_{n+1}))} 
\int_{\mathbb{R}^N} \int_{\mathbb{R}^N} f(x',x) \tilde \pi_{X'X}(x',x|{\rm
  Y}_{n+1}) {\rm d}x' {\rm d}x \\
& \approx \frac{1}{\pi_Y(y_0(t_{n+1}))} 
\sum_{i=1}^M w_i(t_n) f(x_i',x_i) \\
& \propto \sum_{i=1}^M w_i(t_{n+1}) g(x_i') 
\end{align*}
with
\[
f(x',x) = g(x') \frac{\pi_Y(y_0(t_{n+1})|x')
  \pi(x'|x)}{\tilde \pi_{n+1}(x'|x,y_0(t_{n+1}))}
\]
has been used. The guided sequential Monte Carlo method is continued
with $x_i(t_{n+1}) = x_i'$ and new weights $w_i(t_{n+1})$.

Numerical
implementations of guided sequential Monte Carlo methods have been
discussed, for example, in
\cite{sr:leeuwen10,sr:bocquet10,sr:chorin10,sr:chorin12}.  More specifically, a
combined particle and Kalman filter  is proposed in
\cite{sr:leeuwen10} to achieve almost equal particle weights (see also
the discussion in \cite{sr:bocquet10}), while in
\cite{sr:chorin10,sr:chorin12}, new particle positions $x_i(t_{n+1})$
are defined by means of implicit equations. We emphasize that both implementation
approaches give up the requirement of unbiased estimation in hope for reduced variance at
finite ensemble sizes and hence for an overall reduction of the
associated mean squared error (\ref{eq:MSE}).

Another broad class of methods is based on Gaussian mixture
approximations to the prior PDF $\pi_{n+1}(x|{\rm Y}_n)$. Provided
that the forward operator $h$ is linear, the posterior PDF
$\pi_{n+1}(x|{\rm Y}_{n+1})$ is then also a Gaussian mixture and
several procedures have been proposed to adjust the proposals
$x_i^f(t_{n+1})$ such that the adjusted $x_i(t_{n+1})$ approximately
follow the posterior Gaussian mixture PDF. See, for example,
\cite{sr:smith07,sr:frei11,sr:stordal11}. Broadly speaking, these
methods can be understood as providing approximate transport maps
$T_{n+1}'$ instead of an exact transport map $T_{n+1}$ in
(\ref{eq:RR}). However, none of these methods avoid the need for
particle reweighting and resampling. Recall that resampling can be
implemented such that it corresponds to a non-deterministic optimal
transference plan.

The following section is devoted to an embedding technique for
constructing accurate approximations to the transport map $T_{n+1}$ in
(\ref{eq:RR}).

%%%%%%%%%%%%%%%%%%%%%%%%%%%%%%%%%%%%%%%%%%%

\subsection{Continuous ensemble transform filter formulations}

The implementation of (\ref{eq:RR}) requires the computation of a
transport map $T$. Optimal transportation (\emph{i.e.}, maximising the
covariance of the transference plan), leads to $T = \nabla_x \psi$ and
the potential satisfies the highly nonlinear, elliptic Monge-Ampere
equation
\[
\pi_{X_2}(\nabla_x\psi)|D\nabla_x \psi|  = \pi_{X_1} .
\]
A direct numerical implementation for high-dimensional state spaces
${\cal X} = \mathbb{R}^N$ seems at present out of reach. Instead, in
this section we utilize an embedding method due to Moser
\cite{sr:Moser65}, replacing the optimal transport map by a suboptimal
transport map which is defined as the time-one
flow map of a differential equation in artificial time $s\in[0,1]$. At each time
instant, determining the right hand side of the differential equation requires the
solution of a linear elliptic PDE; nonlinearity is exchanged for
linearity at the cost of suboptimality. In some cases, such as
Gaussian PDFs and mixtures of Gaussian, the linear PDE can be solved
analytically. In other cases, further approximations, such as a mean
field approach discussed later in this section, are necessary.

Inspired by the embedding method of Moser \cite{sr:Moser65}, we first
summarize a dynamical systems formulation \cite{sr:reich10} of Bayes'
formula which generalizes the continuous EnKF formulation from Section
\ref{sec:KB}. We first note that a single application of Bayes'
formula (\ref{eq:Bayes}) can be replaced by an $D$-fold recursive
application of the incremental likelihood $\widehat{\pi}$:
\begin{equation} \label{product}
\widehat \pi (y| x) = \frac{1}{(2\pi)^{K/2} |R|^{1/2}} \exp \left(-\frac{1}{2D} \left(
h(x) - y\right)^{\rm T} R^{-1} \left(
h(x) - y\right) \right) ,
\end{equation}
i.e., we first write Bayes formula as
\begin{equation*}
\pi_X (x|y_0) \propto \pi_X(x)\,\prod_{j=1}^D \widehat \pi ( y_0| x),
\end{equation*}
where the constant of proportionality depends only on $y_0$,
and then consider the implied iteration
\begin{equation*}
\pi_{j+1}(x) = \frac{\pi_j(x)\,\widehat \pi ( y_0 | x)}{
\int_\mathbb{R}^N {\rm d} x\, \pi_j(x)\,\widehat \pi ( y_0| x)}
\end{equation*}
with $\pi_0 = \pi_X$ and $\pi_X(\cdot |y_0) = \pi_D$. We may now
expand the exponential function in (\ref{product}) in the small
parameter $\Delta s = 1/D$, in the limit $D\to \infty$ obtaining  the
evolution equation
\begin{equation} \label{alphaODE}
\frac{\partial \pi}{\partial s} = -\frac{1}{2} \left(
h(x) - y_0 \right)^{\rm T} R^{-1} \left(
h(x) - y_0 \right) \pi + \mu \pi
\end{equation}
in the fictitious time $s \in [0,1]$. The scalar Lagrange multiplier $\mu$ is 
equal to the expectation value of the negative log likelihood function 
\begin{equation} \label{loglike}
L( x; y_0) = \frac{1}{2} \left(
h(x) - y_0 \right)^{\rm T} R^{-1} \left(
h(x)- y_0 \right)
\end{equation}
with respect to $\pi$ and ensures that $\int_{\mathbb{R}^N} (\partial
\pi/\partial s)  {\rm d} x  = 0$. 
We also set $\pi(x,0) = \pi_X( x)$ and obtain $\pi_X(x|y_0) = \pi(x,1)$.

We now rewrite (\ref{alphaODE}) in the equivalent, but more compact,
form
\begin{equation} \label{alphaODEII}
\frac{\partial \pi}{\partial s} = -\pi \left( L - \bar{L}\right), \qquad
\mbox{where } \bar{L}= 
 \mathbb{E}_X [L].
\end{equation}
Here $\mathbb{E}_X$ denotes expectation with respect to the PDF $\pi_X
= \pi(\cdot,s)$. It should be noted that the continuous embedding
defined by (\ref{alphaODEII}) is not unique. Moser \cite{sr:Moser65},
for example, used the linear interpolation
\[
\pi(x,s) = (1-s) \pi_X(x) + s\pi_X(x|y_0) ,
\]
which results in
\begin{equation} \label{alphaMoser}
\frac{\partial \pi}{\partial s} = \pi_X(x|y_0) - \pi_X(x).
\end{equation}
Yet another interpolation is given by the displacement interpolation
of McCann which is based on the optimal transportation map and which
has an attractive ``fluid dynamics'' interpretation
\cite{sr:Villani,sr:Villani2}.

Eq.~(\ref{alphaODEII}) (or, alternatively, (\ref{alphaMoser})) defines
the change (or transport) of the PDF $\pi$ in fictitious time $s \in
[0,1]$. Alternatively, following Moser's work
\cite{sr:Moser65,sr:Villani}, we can view this change as being induced
by a continuity (Liouville) equation
\begin{equation} \label{optimaltransportation1}
\frac{\partial \pi}{\partial s} = -\nabla_x \cdot \left(\pi g  \right)
\end{equation}
for an appropriate vector field $g( x,s) \in \mathbb{R}^N$. 

At any time $s \in [0,1]$ the vector field $g(\cdot,s)$ is not
uniquely determined by (\ref{alphaODEII}) and
(\ref{optimaltransportation1}) unless we also require that it is the
minimizer of the kinetic energy
\begin{equation*}
{\cal T}(v) = \frac{1}{2} \int_{\mathbb{R}^N} \pi   v^{\rm T} {\rm M}^{-1} v\,{\rm d}x
\end{equation*}
over all admissible vector fields $v:\mathbb{R}^N \to \mathbb{R}^N$
(\emph{i.e.} $g$ satisfies (\ref{optimaltransportation1}) for given
$\pi$ and $\partial \pi/\partial s$), where ${\rm M} \in \mathbb{R}^{N
  \times N}$ is a positive definite matrix.  Under these
assumptions, minimization of the functional
\begin{equation*}
{\cal L}[v,\phi] = \frac{1}{2} \int_{\mathbb{R}^N} \pi   v^{\rm T} {\rm M}^{-1} v\,{\rm
  d}x + \int_{\mathbb{R}^N} \phi \left\{ 
\frac{\partial \pi}{\partial s} + \nabla_{\bf x} \cdot \left(\pi v  \right) \right\}{\rm d}x
\end{equation*}
for given $\pi$ and $\partial \pi/\partial s$ leads to the
Euler-Lagrange equations
\begin{equation*}
\pi {\rm M}^{-1} g - \pi \nabla_x \psi = 0, \qquad 
\frac{\partial \pi}{\partial s} + \nabla_x \cdot \left(\pi g  \right) = 0
\end{equation*}
in the velocity field $g$ and the potential $\psi$. Hence, provided
that $\pi >0$, the desired vector field is given by $g = {\rm M}
\nabla_x \psi$, and we have shown the following result.

\begin{Proposition*}[Transport map from gradient flow]
If the potential $\psi(x,s)$ is the solution of the elliptic PDE
\begin{equation} \label{optimaltransportation2}
\nabla_x \cdot \left( \pi_X {\rm M} \nabla_x \psi \right) = \pi_X \left( L - \bar L \right),
\end{equation}
then the desired transport map $x' = T(x)$ for the random variable $X$
with PDF $\pi_X(x,s)$ is defined by the time-one-flow map of the
differential equations
\[
\frac{{\rm d}x}{{\rm d}s} = -{\rm M} \nabla_x \psi .
\]
The continous Kalman-Bucy filter equations correspond to the special
case ${\rm M} = P$ and $\psi = \delta F/\delta \pi_X$ with the
functional $F$ given by (\ref{eq:KBFunctional}).
\end{Proposition*}
The elliptic PDE (\ref{optimaltransportation2}) can be solved
analytically for Gaussian approximations to the PDF $\pi_X$ and the
resulting differential equations are equivalent to the ensemble
transform Kalman-Bucy equations (\ref{eq:BR1}). Appropriate analytic
expressions can also be found in case where $\pi_X$ can be
approximated by a Gaussian mixture and the forward operator $h(x)$ is
linear (see \cite{sr:reich11} for details). 

Gaussian mixtures are contained in the class of \emph{kernel
  smoothers}.  It should however be noted that approximating a PDF
$\pi_X$ over high-dimensional phase spaces ${\cal X} = \mathbb{R}^N$
using kernel smoothers is a challenging task, especially if only a relatively small
number of realizations $x_i$, $i=1,\ldots,M$, from the associated
random variable $X$ are available.

In order to overcome this curse of dimensionality, we outline a
modification to the above continuous formulation, which is inspired by
the rank histogram filter of Anderson \cite{sr:anderson10}.  For
simplicity of exposition, consider a single observation $y \in
\mathbb{R}$ with forward operator $h:\mathbb{R}^N \to \mathbb{R}$. We
augment the state vector $x\in \mathbb{R}^N$ by $y = h(x)$, \emph{i.e.}~we
consider $(x,y)$ and introduce the associated joint PDF
\[
\pi_{XY} (x,y) = \pi_X(x|y) \pi_Y(y) .
\] 
We apply the embedding technique first to $y$ alone resulting in
\[
\frac{{\rm d}y}{{\rm d}s} = f_y(y,s)
\]
with
\[
\partial_y (\pi_Y(y) f_y(y)) = \pi_Y(y) (L-\bar L) .
\]
One then finds an equation in the state variable $x \in \mathbb{R}^N$ from
\[
\nabla_x \cdot (\pi_X(x|y) f_x(x,y,s)) + f_y(y,s) \partial_y \pi_X(x|y)  = 0
\]
and
\[
\frac{{\rm d}x}{{\rm d}s} = f_x(x,y,s).
\]
Next we introduce the \emph{mean field approximation}
\begin{equation} \label{eq:meanfield}
\pi_1(x^1|y) \pi_2(x^2|y) \cdots \pi_N(x^N|y)
\end{equation}
for the conditional PDF $\pi_X(x|y)$ with the components of the state vector written as 
$x = (x^1,x^2,\ldots,x^N)^T \in \mathbb{R}^N$. Under the mean field approximation the vector
field $f_x = (f_{x^1},f_{x^2},\ldots,f_{x^N})^T$ can be obtained
component-wise by solving scalar equations
\begin{equation} \label{eq:meanfield2}
\partial_{z} (\pi_k(z|y) f_{x^k}(z,y)) +  f_y(y) \,\partial_y \pi_k(z|y) = 0,
\end{equation}
$k=1,\ldots,N$, for $f_{x^k}(z,y)$ with $z = x^k \in \mathbb{R}$. The (two-dimensional) 
conditional PDFs $\pi_k(x^k|y)$ need to be estimated from the available ensemble members 
$x_i \in \mathbb{R}^N$ by either using parametric or non-parametric statistics.  

We first discuss the case for which both the prior and the
posterior distributions are assumed to be Gaussian.
In this case, the resulting update equations in $x \in \mathbb{R}^N$
become equivalent to the ensemble transform
Kalman-Bucy filter. This can be seen by first noting that
the update in a scalar observable $y \in \mathbb{R}$ is
\[
\frac{{\rm d}y}{{\rm d}s} = -\frac{1}{2} \sigma_{yy}^2 R^{-1} (y+\bar y - 2y_0).
\]
Furthermore, if the condition PDF $\pi_k (z|y)$, $z=x^k \in \mathbb{R}$, 
is of the form (\ref{Gcond1}), then (\ref{eq:meanfield2}) leads to
\[
f_{x^k}(x^k,y) = \sigma_{xy}^2 \sigma_{yy}^{-2} f_y(y),
\]
which, combined with the approximation (\ref{eq:meanfield}), results in the continuous ensemble transform 
Kalman-Bucy filter formulation discussed previously. 

The rank histogram filter of Anderson \cite{sr:anderson10} corresponds in this continuous embedding 
formulation to choosing a general PDF $\pi_Y(y)$ while a Gaussian approximation is used for the 
conditional PDFs $\pi_k(x^k|y)$.

Other ensemble transform filters can  be derived by using appropriate approximations to the marginal
PDF $\pi_Y$ and the conditional PDFs $\pi_k(x^k|y)$, $k=1,\ldots,N$, from the available ensemble 
members $x_i$,  $i=1,\ldots,M$.

%%%%%%%%%%%%%%%%%%%%%%%%%%%%%%%%%%%%%%%%%%%%%%%%%%%%%%%%%%%%

\subsection*{References} 
An excellent introduction to filtering and Bayesian data assimilation is \cite{sr:jazwinski}. 
The linear filter theory (Kalman filter) can, for
example, be found in \cite{sr:simon}.
Fundamental issues of data assimilation in a meteorological context
are covered in \cite{sr:kalnay}.
Ensemble filter techniques and the ensemble Kalman filter are treated
in depth in \cite{sr:evensen}.  Sequential Monte Carlo methods are discussed in
\cite{sr:Doucet,sr:crisan,sr:arul02} and by
\cite{sr:leeuwen,sr:bocquet10} in a geophysical context. See also the recent
monograph \cite{sr:majda}. The transport view has been proposed in
\cite{sr:crisan10} for continuous filter problems and in
\cite{sr:reich10} for intermittent data assimilation.
Gaussian mixtures are a special class of non-parametric kernel smoothing 
techniques which are discussed, for example, in \cite{sr:Wand}.

%%%%%%%%%%%%%%%%%%%%%%%%%%%%%%%%%%%%

\section{Concluding remarks}

We have summarized the Bayesian perspective on sequential data
assimilation and filtering in particular. Special emphasize has been
put on discussing Bayes' formula in the context of coupling of random
variables, which allows for a dynamical system's interpretation of the
data assimilation step.  Within a Bayesian framework all variables are
treated as random. While this implies an elegant mathematical
treatment of data assimilation problems, any Bayesian approach should
be treated with caution in the presence of sparse data,
high-dimensional model problems, and limited sample sizes. It should
be noted in this context that successful assimilation techniques such
as 4DVar (not covered in this survey) and the EnKF lead to biased
approximations to the state estimation problem. In both cases the bias
is due to the fact that the algorithms are derived under the
assumption that the prior distributions are Gaussian. Nevertheless
4DVar and EnKF work often well in terms of the observed mean squared
error (\ref{eq:MSE}) since the variance of the estimator remains small
even for relatively small ensemble sizes $M$. On the contrary,
asymptotically unbiased Bayesian approaches such as sequential Monte
Carlo methods suffer from the curse of dimensionality, lead generally
to large variances in the estimators for small $M$ and have therefore
not yet found systematic applications in operational forecasting, for
example. To overcome this limitation, one could consider more suitable
proposal steps such as guided sequential Monte Carlo methods and/or
impose certain independence assumptions such as mean field
approximations which lead to an improved balance between bias and
variance in the mean squared error (\ref{eq:MSE}). See also the
discussion of \cite{sr:hastie} on the bias-variance tradeoff in the
context of supervised learning. Promising results for guided particle
filters have been reported very recently in
\cite{sr:chorin12b,sr:leeuwen12}.  Alternatively, non-Bayesian
approaches to data assimilation could be explored in the future such
as: (i) shadowing for partially observed reference solutions, (ii) a
nonlinear control approach with transport maps as dynamic feedback
laws, (iii) derivation and analysis of ensemble filter techniques
within the framework of stochastic interacting particle systems.
 
%\include{chapter02}

% Please choose one of the following options:
% VERSION A is for BIBTeX application. You have to
% add the names of the BST file and the databases, e.g.:
\bibliographystyle{plain}
\bibliography{survey}

\begin{thebibliography}{10}

\bibitem{sr:akir11}
J.~Amezcua, E.~Kalnay, K.~Ide, and S.~Reich.
\newblock Using the {K}alman-{B}ucy filter in an ensemble framework.
\newblock {\em submitted}, 2012.

\bibitem{sr:anderson10}
J.L. Anderson.
\newblock A non-{G}aussian ensemble filter update for data assimilation.
\newblock {\em Monthly Weather Review}, 138:4186--4198, 2010.

\bibitem{sr:arul02}
M.S. Arulampalam, S.~Maskell, N.~Gordon, and T.~Clapp.
\newblock A tutorial on particle filters for online nonlinear/non-{G}aussian
  {B}ayesian tracking.
\newblock {\em IEEE Trans. Sign. Process.}, 50:174--188, 2002.

\bibitem{sr:crisan}
A.~Bain and D.~Crisan.
\newblock {\em Fundamentals of stochastic filtering}, volume~60 of {\em
  Stochastic modelling and applied probability}.
\newblock Springer-Verlag, New-York, 2009.

\bibitem{sr:br10b}
K.~Bergemann and S.~Reich.
\newblock A mollified ensemble {K}alman filter.
\newblock {\em Q. J. R. Meteorological Soc.}, 136:1636--1643, 2010.

\bibitem{sr:br11}
K.~Bergemann and S.~Reich.
\newblock An ensemble {K}alman-{B}ucy filter for continuous data assimilation.
\newblock {\em Meteorolog.~Zeitschrift}, 21:213--219, 2012.

\bibitem{sr:bocquet10}
M.~Bocquet, C.A. Pires, and L.~Wu.
\newblock Beyond {G}aussian statistical modeling in geophysical data
  assimilaition.
\newblock {\em Mon.~Wea.~Rev.}, 138:2997--3022, 2010.

\bibitem{sr:StochProc}
Z.~Bre\'zniak and T.~Zastawniak.
\newblock {\em Basic Stochastic Processes}.
\newblock Springer-Verlag, London, 1999.

\bibitem{sr:burgers98}
G.~Burgers, P.J. van Leeuwen, and G.~Evensen.
\newblock On the analysis scheme in the ensemble {K}alman filter.
\newblock {\em Mon. Wea. Rev.}, 126:1719--1724, 1998.

\bibitem{sr:Chorin}
A.~Chorin and O.H. Hald.
\newblock {\em Stochastic tools in mathematics and science}.
\newblock Springer-Verlag, Berlin Heidelberg New York, 2nd edition, 2009.

\bibitem{sr:chorin10}
A.J. Chorin, M.~Morzfeld, and X.~Tu.
\newblock Implicit filters for data assimilation.
\newblock {\em Comm.~Appl.~Math.~Comp.~Sc.}, 5:221--240, 2010.

\bibitem{sr:crisan10}
D.~Crisan and J.~Xiong.
\newblock Approximate {M}c{K}ean-{V}lasov representation for a class of
  {SPDE}s.
\newblock {\em Stochastics}, 82:53--68, 2010.

\bibitem{sr:Doucet}
A.~Doucet, N.~de~Freitas, and N.~Gordon (eds.).
\newblock {\em Sequential Monte Carlo methods in practice}.
\newblock Springer-Verlag, Berlin Heidelberg New York, 2001.

\bibitem{sr:duncan72}
D.B. Duncan and S.D. Horn.
\newblock Linear dynamic recursive estimation from the viewpoint of regression
  analysis.
\newblock {\em J. American Stat. Association}, 67:815--821, 1972.

\bibitem{sr:evensen}
G.~Evensen.
\newblock {\em Data assimilation. {T}he ensemble Kalman filter}.
\newblock Springer-Verlag, New York, 2006.

\bibitem{sr:frei11}
M.~Frei and H.R. K\"unsch.
\newblock Mixture ensemble {K}alman filters.
\newblock {\em Computational Statistics and Data Analysis}, in press, 2011.

\bibitem{sr:gardiner}
C.W. Gardiner.
\newblock {\em Handbook on stochastic methods}.
\newblock Springer-Verlag, 3rd edition, 2004.

\bibitem{sr:golub}
G.H. Golub and Ch.F.~Van Loan.
\newblock {\em Matrix computations}.
\newblock The Johns Hopkins University Press, Baltimore, 3rd edition, 1996.

\bibitem{sr:majda}
J.~Harlim and A.~Majda.
\newblock {\em Filtering Complex Turbulent Systems}.
\newblock Cambridge University Press, Cambridge, 2012.

\bibitem{sr:hastie}
T.~Hastie, R.~Tibshirani, and J.~Friedman.
\newblock {\em The Elements of Statistical Learning}.
\newblock Springer-Verlag, New York, 2nd edition, 2009.

\bibitem{sr:Higham}
D.J. Higham.
\newblock {A}n algorithmic introduction to numerical simulation of stochastic
  differential equations.
\newblock {\em SIAM Review}, 43:525--546, 2001.

\bibitem{sr:jazwinski}
A.H. Jazwinski.
\newblock {\em Stochastic processes and filtering theory}.
\newblock Academic Press, New York, 1970.

\bibitem{sr:Julier97anew}
Simon~J. Julier and Jeffrey~K. Uhlmann.
\newblock A new extension of the kalman filter to nonlinear systems.
\newblock In {\em Signal processing, sensor fusion, and target recognition.
  Conference No.~6}, volume 3068, pages 182--193, Orlando FL, 1997.

\bibitem{sr:kaipio}
J.~Kaipio and E.~Somersalo.
\newblock {\em Statistical and computational inverse problems}.
\newblock Springer-Verlag, New York, 2005.

\bibitem{sr:kalnay}
E.~Kalnay.
\newblock {\em Atmospheric modeling, data assimilation and predictability}.
\newblock Cambridge University Press, 2002.

\bibitem{sr:Kloeden}
P.E. Kloeden and E.~Platen.
\newblock {\em Numerical solution of stochastic differential equations}.
\newblock Springer-Verlag, Berlin Heidelberg New York, 1992.

\bibitem{sr:leeuwen}
P.J.~Van Leeuwen.
\newblock Particle filtering in geophysical systems.
\newblock {\em Monthly Weather Review}, 137:4089--4114, 2009.

\bibitem{sr:leeuwen10}
P.J.~Van Leeuwen.
\newblock Nonlinear data assimilation in the geosciences: an extremely
  efficient particle filter.
\newblock {\em Q.J.R.~Meteorolog.~Soc.}, 136:1991--1996, 2010.

\bibitem{sr:leeuwen12}
P.J.~Van Leeuwen and M.~Ades.
\newblock Efficient fully nonlinear data assimilation for geophysical fluid
  dynamics.
\newblock {\em Computers and Geosciences}, 47:in press, 2012.

\bibitem{sr:lei11}
J.~Lei and P.~Bickel.
\newblock A moment matching ensemble filter for nonlinear and non-{G}aussian
  data assimilation.
\newblock {\em Mon.~Weath.~Rev.}, 139:3964--3973, 2011.

\bibitem{sr:Lewis}
J.M Lewis, S.~Lakshmivarahan, and S.K. Dhall.
\newblock {\em Dynamic data assimilation: {A} least squares approach}.
\newblock Cambridge University Press, Cambridge, 2006.

\bibitem{sr:Liu}
J.S. Liu.
\newblock {\em {M}onte {C}arlo Strategies in Scientific Computing}.
\newblock Springer-Verlag, New York, 2001.

\bibitem{sr:Meyn}
S.P. Meyn and R.L. Tweedie.
\newblock {\em Markov chains and stochastic stability}.
\newblock Springer-Verlag, London New York, 1993.

\bibitem{sr:chorin12b}
M.~Morzfeld and A.J. Chorin.
\newblock Implicit particle filtering for models with partial noise and an
  application to geomagnetic data assimilation.
\newblock {\em Nonlinear Processes in Geophysics}, 19:365--382, 2012.

\bibitem{sr:chorin12}
M.~Morzfeld, X.~Tu, E.~Atkins, and A.J. Chorin.
\newblock A random map implementation of implicit filters.
\newblock {\em J.~Comput.~Phys.}, 231:2049--2066, 2012.

\bibitem{sr:marzouk11}
T.A.~El Moselhy and Y.M. Marzouk.
\newblock Bayesian inference with optimal maps.
\newblock {\em J. Comput. Phys.}, 231:in press, 2012.

\bibitem{sr:Moser65}
J.~Moser.
\newblock On the volume elements on a manifold.
\newblock {\em Trans.~Amer.~Math.~Soc.}, 120:286--294, 1965.

\bibitem{sr:Neal}
R.M. Neal.
\newblock {\em {B}ayesian learning for neural networks}.
\newblock Springer-Verlag, New York, 1996.

\bibitem{sr:SDEbook}
B.~{\O}ksendal.
\newblock {\em Stochastic Differential Equations}.
\newblock Springer-Verlag, Berlin-Heidelberg, 5th edition, 2000.

\bibitem{sr:olkin82}
I.~Olkin and F.~Pukelsheim.
\newblock The distance between two random vectors with given dispersion
  matrices.
\newblock {\em Linear Algebra and its Applications}, 48:257--263, 1982.

\bibitem{sr:Otto01}
F.~Otto.
\newblock The geometry of dissipative evolution equations: the porous medium
  equation.
\newblock {\em Comm.~Part.~Diff.~Eqs.}, 26:101--174, 2001.

\bibitem{sr:reich10}
S.~Reich.
\newblock A dynamical systems framework for intermittent data assimilation.
\newblock {\em BIT Numer Math}, 51:235--249, 2011.

\bibitem{sr:reich11}
S.~Reich.
\newblock A {G}aussian mixture ensemble transform filter.
\newblock {\em Q.~J.~R.~Meterolog.~Soc.}, 138:222--233, 2012.

\bibitem{sr:Roberts96}
G.O. Roberts and R.L. Tweedie.
\newblock Exponential convergence of {L}angevin distributions and their
  discrete approximations.
\newblock {\em Bernoulli}, 2:341--363, 1996.

\bibitem{sr:simon}
D.J. Simon.
\newblock {\em Optimal state estimation}.
\newblock John Wiley \& Sons, Inc., New York, 2006.

\bibitem{sr:smith07}
K.W. Smith.
\newblock Cluster ensemble {K}alman filter.
\newblock {\em Tellus}, 59A:749--757, 2007.

\bibitem{sr:stordal11}
A.S. Stordal, H.A. Karlsen, G.~N\ae{}vdal, H.J. Skaug, and B.~Vall\'es.
\newblock Bridging the ensemble {K}alman filter and particle filters: the
  adaptive {G}aussian mixture filter.
\newblock {\em Comput. Geosci.}, 15:293--305, 2011.

\bibitem{sr:stuart10a}
A.M. Stuart.
\newblock Inverse problems: a bayesian perspective.
\newblock In {\em Acta Numerica}, volume~17. Cambridge University Press,
  Cambridge, 2010.

\bibitem{sr:Villani}
C.~Villani.
\newblock {\em Topics in Optimal Transportation}.
\newblock American Mathematical Society, Providence, Rhode Island, NY, 2003.

\bibitem{sr:Villani2}
C.~Villani.
\newblock {\em Optimal transportation: {O}ld and new}.
\newblock Springer-Verlag, Berlin Heidelberg, 2009.

\bibitem{sr:Wand}
M.P. Wand and M.C. Jones.
\newblock {\em Kernel smoothing}.
\newblock Chapmann and Hall, London, 1995.

\bibitem{sr:xiong06}
X.~Xiong, I.M. Navon, and B.~Uzungoglu.
\newblock A note on the particle filter with posterior {G}aussian resampling.
\newblock {\em Tellus}, 85A:456--460, 2006.

\end{thebibliography}

% VERSION B is the LaTeX standard bibliography environment:
%\begin{thebibliography}{99}
%  \bibitem{Argyros:2007} S. A. Argyros and P. Dodos,
%    \emph{Genericity and amalgamation of classes of  Banach spaces},
%    Adv. Math. 209 (2007), pp.~666--748.
%\end{thebibliography}

% Specify the content of the index by using \index commands. Then run
% makeindex main.tex -s degruyter-mono.ist
% to create the index.
%\printindex

\end{document}